\documentclass[onecolumn,12pt]{article}
\usepackage{epsfig,psfrag}

\usepackage{graphicx}
\usepackage{amsmath}
\usepackage{mathtools}
\usepackage{amssymb}
\usepackage{color}
\usepackage{bbm}
\usepackage{mathrsfs}
\usepackage{dsfont,accents}
\usepackage{fullpage}
\usepackage[sort, compress]{cite}

\everymath{\displaystyle}

\newtheorem{theorem}{Theorem}%[\arabic{section}]

\newtheorem{proposition}[theorem]{Proposition}
\newtheorem{lemma}[theorem]{Lemma}
\newtheorem{define}[theorem]{Definition}
\newtheorem{example}[theorem]{Example}
\newtheorem{assumption}[theorem]{Assumption}
\newtheorem{conjecture}[theorem]{Conjecture}
\newtheorem{remark}[theorem]{Remark}
\usepackage{enumitem}
\setlist[enumerate]{leftmargin=.5in}
\setlist[itemize]{leftmargin=.5in}

\DeclareMathOperator*{\col}{col}
\DeclareMathOperator*{\row}{row}

\DeclareMathOperator*{\diag}{diag}
\DeclareMathOperator*{\rank}{rank}
\DeclareMathOperator{\eps}{\varepsilon}

\def\P{\mathbb{P}}
\def\E{\mathbb{E}}

\def\PH{\textnormal{PH}}

\providecommand\X[1]{\boldsymbol{X_{#1}}}
\providecommand\Z[1]{\boldsymbol{Z_{#1}}}

\providecommand\D[1]{\boldsymbol{D_{#1}}}

\providecommand\phib{\boldsymbol{\emptyset}}

\def\lllongrightarrow{\relbar\joinrel\relbar\joinrel\relbar\joinrel\relbar\joinrel\relbar\joinrel\relbar\joinrel\rightarrow}

\providecommand{\rarrow}[1]{\stackrel{#1}{\lllongrightarrow}}

\definecolor{lblue}{RGB}{0,110,152}

\definecolor{dred}{RGB}{171,67,53}

\newenvironment{proof}{{\it Proof :~}}{\hfill$\diamondsuit$\\}

\begin{document}

\title{Ergodicity analysis and antithetic integral control of a class of stochastic reaction networks with delays}

\author{Corentin Briat\thanks{D-BSSE, ETH-Z\"urich, Switzerland
  (\textit{corentin@briat.info}, www.briat.info).}
and Mustafa Khammash\thanks{D-BSSE, ETH-Z\"urich, Switzerland (\textit{mustafa.khammash@bsse.ethz.ch}, https://www.bsse.ethz.ch/ctsb.)}}%\footnotemark[3]}

\date{}

\maketitle

\begin{abstract}
Delays are important phenomena arising in a wide variety of real world systems, including biological ones, because of diffusion/propagation effects or as simplifying modeling elements. We propose here to consider delayed stochastic reaction networks, a class of networks that has been relatively few studied until now. The difficulty in analyzing them resides in the fact that their state-space is infinite-dimensional. We demonstrate here that by restricting the delays to be phase-type distributed, one can represent the associated delayed reaction network as a reaction network with finite-dimensional state-space. This can be achieved by suitably adding chemical species and reactions to the delay-free network following a simple algorithm which is fully characterized. Since phase-type distributions are dense in the set of probability distributions, they can approximate any distribution arbitrarily closely and this makes their consideration only a bit restrictive. As the state-space remains finite-dimensional, usual  tools developed for non-delayed reaction network directly apply. In particular, we prove, for unimolecular mass-action reaction networks,  that the delayed stochastic reaction network is ergodic if and only if the delay-free network is ergodic as well. Bimolecular reactions are more difficult to consider but slightly stronger analogous results are nevertheless obtained. These results demonstrate that delays have little to no harm to the ergodicity property of reaction networks as long as the delays are phase-type distributed, and this holds regardless the complexity of their distribution. We also prove that the presence of those delays adds convolution terms in the moment equation but does not change the value of the stationary means compared to the delay-free case. The covariance, however, is influenced by the presence of the delays. Finally, the control of a certain class of delayed stochastic reaction network using a delayed antithetic integral controller is considered. It is proven that this controller achieves its goal provided that the delay-free network satisfy the conditions of ergodicity and output-controllability.\\

\end{abstract}

\section{Introduction}

Delays are omnipresent physical phenomena induced by memory, propagation or transport effects \cite{Kolma:92,Niculescu:01,GuKC:03,Michiels:07bk,Briat:book1,Fridman:14}. They naturally arise in population dynamics \cite{Gopalsamy:92}, ecology \cite{Gopalsamy:88}, epidemiology \cite{Briat:09h}, biology \cite{Bratsun:05,Jansen:15,Fridman:16,Djema:17,Djema:18} and engineering \cite{Kolma:86,Niculescu:01,GuKC:03,Briat:book1}. It is commonly understood that delays have, in general, detrimental effects in engineering as they may lead to instabilities such as oscillations. While this destabilizing effect is undesirable in this setting, their role can be crucial in biology when one wants, for instance, to design oscillators \cite{Bratsun:05,Stricker:08,Mather:09}. In the stochastic setting, delays have indeed been shown to be helpful for generating oscillations \cite{Bratsun:05}, but also to accelerate signaling \cite{Josic:11} and to be responsible for an increase in intrinsic variability \cite{Scott:07}. Delays can be easily incorporated in the dynamics of a deterministic reaction network by simply substituting delay-free terms by delayed ones, thereby turning ordinary differential equations into delay-differential equations, the analysis of which can be readily carried out using well-developed techniques such as Lyapunov-based ones or input-output methods; see e.g. \cite{Kolma:92,Niculescu:01,GuKC:03,Michiels:07bk,Briat:book1}. When the dynamics of the reaction network is inherently stochastic and represented by a continuous-time jump Markov process \cite{Anderson:11}, the introduction of constant deterministic delays in the dynamics is also possible. Those networks can be easily simulated using a simple adaptation of  Gillespie's stochastic simulation algorithm \cite{Bratsun:05}, the next reaction method \cite{Anderson:07} or using delayed continuous-time Markov chains \cite{Guet:12}. Alternatively, it has been shown in \cite{Barrio:13,Leier:14,Falk:18} that certain chains of unimolecular reactions in a reaction network could be equivalently substituted by stochastic time-varying delays whose distributions can exactly computed from the reaction rates. This has led to a drastic speedup in the stochastic simulations.

When stochastic reaction networks modeled as jump Markov processes are considered, it has been shown that the notion ergodicity is a natural notion of stability that can be established using algebraic, graph theoretical and optimization techniques \cite{Briat:13i,Briat:14d,Briat:16cdc,Briat:17ifacBio,Briat:17LAA,Gupta:18}. Ergodicity is the stochastic analogue of having a unique globally attractive fixed point for deterministic dynamics. It can be used to establish moment convergence as well as the property that the population behavior can be deduced from a single trajectory of the Markov process. Checking whether a stochastic reaction network is ergodic amounts to establishing two properties: the irreducibility of the state-space (or a particular subset of it) and the fulfilment of a Foster-Lyapunov condition. While it is quite clear how these conditions could be checked for standard (i.e. non-delayed) stochastic reaction networks, the case of delayed stochastic reaction networks is more complicated. Indeed, since the state-space of a general delayed reaction network is infinite-dimensional, checking the irreducibility of a function-space is way more involved. The Foster-Lyapunov condition which is based on the use of a norm-like function is not easy to generalize but tools from time-delay systems theory, such as Lyapunov-Krasovskii functionals, may need turn out to be useful in this task.

The objective of this paper is to develop a framework for the modeling, the analysis and the control of reaction networks with stochastically time-varying delays. However, our goal is to avoid the consideration of a Markov jump system with infinite-dimensional state-space and remain in the finite-dimensional case. Interestingly, this can be done by assuming that the delays follow a phase-type distribution, a class of distribution arising, for instance, in queueing networks \cite{HarcholBalter:13}, risk theory \cite{Bladt:05}, health-care \cite{Shaw:07} and evolution \cite{Strimmer:01}. As those distributions are dense (in the sense of weak convergence) in the set of all probability distributions on $(0,\infty)$ \cite{Asmussen:03}, they can be used to approximate any delay-distribution with arbitrary precision. In this regard, this class of distributions are theoretically mildly restrictive if we allow those distributions to have an arbitrarily large number of free parameters. Several algorithms are available for approximating a given distribution or for fitting empirical ones; see e.g. \cite{Asmussen:96}. The reason behind the use of such distributions is that they can be exactly represented as an irreducible unimolecular reaction network, meaning that delays can be included in the network by suitably adding extra species and extra reactions, thereby creating an augmented network with a finite, yet possibility high-, dimensional state-space. A procedure for characterizing distributions that can be represented as phase-type ones is proposed along with a constructive method for representing the phase-type distributed delay as a minimal reaction network. Such a minimal network is notably shown to be unique. As a consequence, since the state-space remains finite-dimensional, existing tools can be applied to the augmented network to yield results on delayed reaction networks having delays that are phase-type distributed.

We propose to use the tools developed in \cite{Briat:13i,Briat:15e} in order to establish several results for delayed reaction networks. Using the ergodicity results developed in \cite{Briat:13i}, we prove that a delayed unimolecular network is ergodic if and only if its delay-free counterpart is ergodic as well. This result is interesting for two reasons. The first one is that phase-type distributed delays are harmless in the context of unimolecular networks. The second one is that the network will remain ergodic for any phase-type distributed delays, regardless of the complexity of their distribution. This includes complex distributions approximating arbitrarily closely heavy-tailed distributions or Dirac distributions. In this regard, the computational complexity of checking whether a delayed reaction network is ergodic does not increase as the delay-distribution increases in complexity but only depends on the number of nominal molecular species involved in the delayed network. This result parallels existing ones on the stability of linear positive systems with delays for which it is well-known that the stability is equivalent to that of the delay-free system and is independent of the value of the delay; see e.g. \cite{Haddad:04,Briat:16b}.

Similar results are obtained for bimolecular networks. When there is no delayed bimolecular reactions, it is shown that the ergodicity conditions reduces to those of the delay-free network. So, in this case again, one can see that the delays are not detrimental to the ergodicity of the network and that the complexity of verifying whether a delayed reaction network satisfying those conditions is the same as checking the ergodicity of a bimolecular network. When some delayed bimolecular reactions are present, the conditions of the delay-free case are not fully recovered but only a slight variations of them, not necessarily more restrictive. In this case also, the results can be associated with an ergodicity test for a delay-free network having a lower computational complexity than the augmented one.

It is then shown that phase-type distributed delays introduce convolution terms in the dynamics of the species of the delayed reaction network with kernels corresponding to the distributions of the delays. Interestingly, we also prove, under the ergodicity assumption, that the mean stationary value coincides with that of the delay-free network.

Finally, we address the problem of controlling delayed reaction networks using the antithetic integral controller proposed in \cite{Briat:15e}. We notably generalize the controller to include delayed reactions in the actuation and the measurement parts of the controller. We show that the delayed reaction network satisfy the ergodicity and output-controllability conditions if and only if the delay-free reaction satisfies the very same conditions. This result parallels those obtained for the ergodicity analysis. In this regard, if the delay-free network verifies the ergodicity and output-controllability conditions then the delayed network will also verify them, regardless of the complexity of the delay distributions.\\

\noindent\textbf{Outline.} Some preliminaries on reaction networks are first given in Section \ref{sec:preliminaries}. Phase-type distributed delays are introduced in Section \ref{sec:delays} and fully characterized in terms of algebraic conditions. A constructive procedure for building the associated reaction network is also provided. Section \ref{sec:delayednetworks} introduces delayed reaction with phase-type distributed delays. Ergodicity conditions for those networks are provided in Section \ref{sec:ergodicity}. The associated moments equation is briefly studied in Section \ref{sec:moments}. Finally, conditions ensuring their control using an antithetic integral controller are obtained in Section \ref{sec:AIC}. Concluding discussions are provided in Section \ref{sec:conclusion}.\\

\noindent\textbf{Notations.} The cones of positive and nonnegative $d$-dimensional vectors are denoted by $\mathbb{R}^d_{>0}$ and $\mathbb{R}^d_{\ge0}$, respectively, whereas the set of nonnegative integers is denoted by $\mathbb{Z}_{\ge0}$. The vector $\mathds{1}$ is the vector of ones. The operators $\textstyle\diag_i(x_i)=\diag(x_1,\ldots,x_n)$, $\textstyle\col_i(x_i)=\col(x_1,\ldots,x_n)$ and $\textstyle\row_i(x_i)=\row(x_1,\ldots,x_n)$  denote the matrices consisting of placing the elements diagonally, vertically and horizontally, respectively. A square matrix $M\in\mathbb{R}^{d\times d}$ is said to be Hurwitz stable if all its eigenvalues have negative real part. The matrix $M$ is said to be Metzler if all its off-diagonal elements are nonnegative. The standard basis of $\mathbb{R}^n$ is denoted by the vectors $e_1,\ldots,e_n$.

\section{Preliminaries on stochastic reaction networks}\label{sec:preliminaries}

\subsection{Stochastic reaction networks without delays}

A reaction network $(\X{},\mathcal{R})$ is a set of $d$ molecular species $\X{}=(\X{1},\ldots,\X{d})$ interacting through $K$ reaction channels $\mathcal{R}:=\{\mathcal{R}_1,\ldots,\mathcal{R}_K\}$. For each of reaction, we denote the stoichiometric vector of the $k$-th reaction by $\zeta_k\in\mathbb{Z}^d$ and the propensity of the $k$-th reaction by $\lambda_k(\cdot)$ where $\lambda_k:\mathbb{Z}_{\ge0}^d\to\mathbb{R}_{\ge0}$ with the additional condition that if $x+\zeta_k\notin\mathbb{Z}_{\ge0}^d$ then $\lambda_k(x)=0$. Under the well-mixed assumption, the process $(X(t))_{t\ge0}=((X_1(t),\ldots,X_d(t))_{t\ge0}$ describing the evolution over time of the molecular counts trajectory is a Markov process. To this Markov process, we associate a state-space $\mathcal{S}$ defined as the subset of $\mathbb{Z}_{\ge0}^d$ that is forward invariant and minimal, that is, it is the smallest set $\mathcal{S}$ such that if $X(0)=x_0\in\mathcal{S}$, then $X_{x_0}(t)\in\mathcal{S}$ for all $t\ge0$. Such a definition is relevant in the context of the study of the ergodicity properties of networks; for more remarks see \cite{Briat:13i}.

Let $\mathcal{P}(\mathcal{S})$ be the set of all probability distributions on the state-space $\mathcal{S}$ which is endowed with the weak topology. We can then define the probability
\begin{equation}
  p_{x_0}(t,x)=\P(X_{x_0}(t)=x)
\end{equation}
where $x_0,x\in\mathcal{S}$. Defining then $p_{x_0}(t)(A):=\textstyle\sum_{y\in A}p_{x_0}(t,y)$ where $A\subset\mathcal{S}$, then $p_{x_0}(t)$ can be understood as an element of $\mathcal{S}$ which coincides, actually, with the distribution of the Markov process $(X(s))_{s\ge0}$ at time $t$. The evolution of $p_{x_0}(t)$ is governed by the Chemical Master Equation (CME, or Forward Kolmogorov Equation) given by
\begin{equation}\label{eq:CME}
  \dfrac{dp_{x_0}(t,x)}{dt}=\sum_{i=1}^K\left[\lambda_k(x-\zeta_k)p_{x_0}(t,x-\zeta_k)-\lambda_k(x)p_{x_0}(t,x)\right].
\end{equation}
where $p(0,x)=\delta(x-x_0)$ is the Kronecker $\delta$ function. In general, the CME is not analytically solvable except in some particular simple cases; see e.g. \cite{Jahnke:07} and the references therein. This is the reason why numerical solutions are of interest; see e.g. \cite{Kazeev:14,Gupta:17}. Alternatively, we can write the so-called random time change representation of the system which takes the form
\begin{equation}
  X(t)=X(0)+\sum_{i=1}^K\zeta_iY_i\left(\int_0^t\lambda_i(X(s))ds\right)
\end{equation}
where the $Y_i$'s are independent unit-rate Poisson processes and $X(s)=0$, $s<0$. Finally, it is important to define the generator $\mathbb{A}$ of the Markov process representing the reaction network $(\X{},\mathcal{R})$:
\begin{equation}\label{eq:generator}
  \mathbb{A}f(x)=\sum_{i=1}^K\lambda_i(x)\left[f(x+\zeta_i)-f(x)\right]
\end{equation}
for all functions $f:\mathbb{Z}_{\ge0}^d\to\mathbb{R}$ in the domain of $\mathbb{A}$. For such functions $f$, Dynkin's formula is valid and we have that
\begin{equation}
  \E[f(X(t))]=\E[f(X(s))]+\int_s^t\E[\mathbb{A}f(X(\theta))]d\theta
\end{equation}
for all $s\le t$.

\subsection{Stochastic reaction networks with delays}

Reaction networks with delays are not new and have been studied in the past, both in the deterministic \cite{Mackey:77,Bliss:82,Glass:88,Gopalsamy:92,Ruan:06} and stochastic \cite{Bratsun:05,Anderson:07,Jansen:15} settings. Let $\tau_k$ be the delay of reaction $k$ and decompose the stoichiometric vector $\zeta_k$ as $\zeta_k=\zeta_k^r-\zeta_\ell^\ell$ where the $\zeta_k^r\in\mathbb{Z}^d_{\ge0}$ and $\zeta_k^\ell\in\mathbb{Z}^d_{\ge0}$ are the right- and left-stoichiometric vector of the $k$-th reaction. Then, we can decompose each delayed reaction as a sequence of two reactions. The first one happens instantaneously, i.e. when the $k$-th reaction fires (note that the propensities always depend on the current state $X(t)$), and changes the state value by $x\mapsto x-\zeta_k^\ell$. The second one occurs after $\tau_k$ seconds and changes the state value by $x\mapsto x+\zeta^r_k$. This temporal behavior is difficult to capture in the chemical master equation or in the generator without extending the state-space. However, this can be easily incorporated in the random time change representation of the system as it is a temporal characterization of the process \cite{Anderson:11}.  That is, we have that
\begin{equation}
  X(t)=X(0)+\sum_{i=1}^K\left[-\zeta_i^\ell Y_i\left(\int_0^t\lambda_i(X(s))ds\right)+\zeta_i^r Y_i\left(\int_0^{t-\tau_i}\lambda_i(X(s))ds\right)\right].
\end{equation}
where the $Y_i$'s are independent unit-rate Poisson processes and $X(s)=0$, $s<0$. Note the coupling of the instantaneous part and the delayed part of the reaction through the same Poisson process. However, depending on the considered network, the above decomposition may not be the best one to represent a delayed reaction as the species on the left-hand side on the reaction are not necessarily destroyed (e.g. in catalytic reactions). This will be further discussed in Section \ref{sec:delayednetworks}.

\subsection{Types of delays}

We discuss here the different types of delays that can be considered. We have to distinguish here two problems: incorporating delays for simulation purposes and incorporating delays for modeling and analysis purposes. The first one is usually easier than the second one as the problem is of computational nature and existing algorithms can be easily adapted to include delayed reactions. The difficulty with the modeling and the analysis problems is that the state-space of a delayed reaction network essentially becomes infinite dimensional and this, therefore, makes irreducibility analysis of the network more difficult.\\

\noindent\textbf{Deterministic and stochastic constant delays.} Certainly the most natural type of delays that comes to mind are deterministic constant delays. As briefly stated in \cite{Anderson:11}, the random time change representation and Gillespie's stochastic simulation algorithm can be adapted to cope with such delays. On the modeling and analysis side, the state-space becomes infinite-dimensional which leads to difficulty for extending the chemical master equation, checking the irreducibility of the state-space and the positive recurrence of the Markov process. The case of stochastic constant delays is analogous. At the beginning of the simulation, the delays are drawn from the distribution and kept constant until the simulation is over. Regarding the modeling and analysis, the problem is the same as for deterministic constant delays.\\

\noindent\textbf{Deterministic time-varying delays.} Existing algorithms should, in principle, be extended to account for deterministic time-varying delays. However, this may lead to a high increase of the computational complexity. Regarding modeling and analysis are more complex than their constant counterparts and, therefore, lead to, at least, the same difficulties. Additionally, the ergodicity of such deterministically time-varying system may not be easy to define.\\

\noindent\textbf{Stochastic time-varying delays.} Stochastic time-varying delays, as we shall see later, are much more natural to consider in this context provided that their distribution is of some particular kind. It has been shown in \cite{Barrio:13,Leier:14} that bidirectional chains of unimolecular conversion reactions could be substituted, in simulation, by a time-varying stochastic delay, the distribution of which being computed from the reaction rates and the topology of the reaction network to be reduced. This has led to dramatic improvements in terms of simulation time. In this regard, those delays should be easily incorporable in the reaction network model and should also facilitate their analysis. This will be shown to be true when the delay distribution belongs to the class of phase-type distributions; see e.g. \cite{Asmussen:03}. It will be notably demonstrated that delayed reaction networks can be exactly reformulated as a non-delayed network with finite-dimensional state-space.

\subsection{Ergodicity analysis of stochastic reaction networks}

The ergodicity of stochastic chemical reaction networks is the important property that the CME has a unique attractive fixed-point. This is formalized in the definition below:
\begin{define}
  Let us consider the stochastic reaction network $(\X{},\mathcal{R})$. The associated Markov process, described by the CME \eqref{eq:CME}, is \textit{ergodic} if  there exists a unique stationary distribution $p^*$ such that $$p_{x_0}(t,x)\to p^*(x)$$ as $t\to\infty$ for all initial conditions $p(0,x)=\delta(x-x_0)$, $x_0\in\mathbb{Z}_{\ge0}^d$.  Moreover, when the convergence is exponentially fast, then the Markov process is \emph{exponentially ergodic}.
\end{define}
Several conditions have been provided in the literature for checking the ergodicity of Markov processes. An interesting approach is based on the so-called Foster-Lyapunov functions introduced in \cite{Meyn:93}. This approach has been specialized to the case of stochastic reaction networks in \cite{Briat:13i,Milias:14}. Additional results on the stability of reaction networks have also been provided in \cite{Engblom:12,Rathinam:13,Anderson:17}. We have the following general result that we specialize to our setup \cite{Meyn:93}:
\begin{theorem}\label{prop:meyn}\
  Let us consider the stochastic reaction network $(\X{},\mathcal{R})$ and assume that its state-space is irreducible. Let $\mathbb{A}$ be the generator of the underlying Markov process as defined in \eqref{eq:generator} and let $C\subset\mathbb{Z}_{\ge0}^d$ be some compact set. Assume further that there exists a function $V:\mathbb{Z}_{\ge0}^d\mapsto\mathbb{R}$, a function $f:\mathbb{Z}_{\ge0}^d\mapsto\mathbb{R}$, $f(x)\ge1$ for all $x\in\mathbb{Z}_{\ge0}^d$ and scalars  $c,d>0$ such that\hspace{1in}
  {\begin{enumerate}[label=({\alph*})]
    \item $V$ is norm-like (i.e. nonnegative and radially unbounded)
    %\item $V$ is bounded on
    \item $\mathbb{A}V(x)\le -cf(x)+d1_{C}(x)$ where $1_{C}$ is the indicator function of the set $C$.
  \end{enumerate}}
  Then, the Markov process is ergodic where the convergence is in total variation.  Moreover, if the above conditions hold with $f(x)=V(x)$, then the Markov process is exponentially ergodic. In any of those cases, the function $V$ is said to be Foster-Lyapunov function.
\end{theorem}
Equivalent conditions for the exponential ergodicity are $V(0)=0$, $V(x)>0$ for all $x\in\mathbb{R}^d$, $V$ radially unbounded and $\mathbb{A}V(x)\le c_1-c_2V(x)$ for some $c_1\ge0,c_2>0$ and for all $x\in\mathbb{Z}_{\ge0}^d$. This reformulation similar to standard Lyapunov conditions will be preferred in this paper for simplicity.

In the case of stochastic mass-action reaction networks with zeroth-, first- and second-order reactions, it is possible to exploit the positivity of the dynamics of the system and consider a linear copositive Foster-Lyapunov function of the form $V(x)=v^Tx$ where $v\in\mathbb{R}^d_{>0}$. Before stating the main result, we need few definitions. We define the stoichiometric matrix $S$ of this network as
\begin{equation}
  S=\begin{bmatrix}
    S_0 & S_u & S_b
  \end{bmatrix}\in\mathbb{Z}^{d\times K}
\end{equation}
 where $S_0$ corresponds to zeroth order reactions, $S_u$ to first order reactions and $S_b$ to second order reactions. Correspondingly, we define the  propensity functions associated with the zeroth- and first-order reactions as $\lambda_0\ge0$ and $\lambda_u(x)=W_ux$ where $W_u$ is an appropriate nonnegative real-valued matrix. We also make the following assumption:
 \begin{assumption}\label{hyp:1}
   The network is open in the sense that it has no closed component and the state-space of the associated Markov process is irreducible.
 \end{assumption}
The assumption on the openness of the network is considered to simplify the exposition of the results but does not make them more restrictive since establishing the ergodicity of closed-components simply follows from the irreducibility of their state-space.
The irreducibility is also assumed here as this is not the main topic of this paper. However, it seems then important to clarify how this can be checked. First of all, note that given a network topology, the irreducibility of the state-space is a structural property in the sense that if it holds for a given network with certain reaction rates, it will also hold for all possible positive values of the reaction rates. Zero values, if considered, will need to be taken care separately in order to make sure that the removal of each of the corresponding reactions do not destroy the irreducibility property of the state-space. It has been proven that the irreducibility of  reaction networks can be checked by solving a linear program and a simple linear algebraic condition \cite{Gupta:18}. In this regard, proving the ergodicity of reaction networks is possible using simple algebraic and computational methods. Note, however, that the conditions are in general sufficient only, but it has been emphasized in \cite{Briat:13i} that they have successfully been able to establish the ergodicity of several typical reaction networks considered in the literature.

We first consider the case unimolecular networks and provide the following result which consists of a slight generalization of the one in  \cite{Briat:13i}:
\begin{theorem}[Unimolecular reaction networks]\label{thm:uni}
   Let us consider stochastic reaction network $(\X{},\mathcal{R})$ with zeroth- and first-order mass-action kinetics which satisfies Assumption \ref{hyp:1}. Then, the following statements are equivalent:
   \begin{enumerate}[label=({\alph*})]
    \item There exists a vector $v\in\mathbb{R}^d_{>0}$ for which the function $V(x)=v^Tx$ is a Foster-Lyapunov function for the corresponding Markov process.
     \item There exists a vector $v\in\mathbb{R}_{>0}^d$ such that $v^TS_uW_u<0$.
     \item The matrix $S_uW_u$ is Hurwitz stable (i.e. all the eigenvalues have negative real-part).
     \item The Markov process describing the reaction network is exponentially ergodic.
   \end{enumerate}
   Moreover, the stationary distribution is light-tailed and, as a consequence, all the moments are bounded and globally converging to their unique stationary value.
\end{theorem}
 \begin{proof}
   To prove the equivalence between the statements (a) and (b), we consider the modified versions of the conditions in Theorem \ref{prop:meyn} with the function $V(x)=v^Tx$, $v>0$. We get that
   \begin{equation}
     \mathbb{A}V(x)=v^T(S_uW_ux+S_0\lambda_0).
   \end{equation}
   Clearly if $v^TS_uW_u<0$, then there exists a $c_2>0$ such that  $\mathbb{A}V(x)\le-cV(x)+c_1$ where $c_1=v^TS_0\lambda_0$.  This proves the equivalence. The equivalence with the statement (c) comes from the fact that $S_uW_u$ is Metzler and that the condition in the statement (b) is a necessary and sufficient condition for its Hurwitz stability. The implication of the statement (d) from the statement (a) follows from Theorem \ref{prop:meyn}. To prove the implication of the statement (c) from the statement (d), first note that the dynamics of the first-order moments is given by
   \begin{equation}
     \dfrac{d\E[X(t)]}{dt}=S_uW_u\E[X(t)]+S_0\lambda_0.
   \end{equation}
   Since the network is (exponentially) ergodic, then the first order moments need to converge to a unique steady-state value and, as the network has no closed components, this can only be the case if the matrix $S_uW_u$ is Hurwitz stable. The conclusion on the light-tailedness follows from \cite{Briat:13i}.
 \end{proof}

 \begin{remark}
   When the network has irreducible closed components, one can simply remove those components from the network and apply the above result to the reduced network. In this case, the matrix $S_uW_u$ will correspond to the matrix describing the interactions between the species in the open part of the network.
 \end{remark}

The case of bimolecular networks is a bit more involved but one existing result shows that their analysis may not be too far off from the analysis of unimolecular networks \cite{Briat:13i}:
\begin{theorem}[Bimolecular reaction networks]\label{thm:bi}
  Let us consider a stochastic reaction network $(\X{},\mathcal{R})$ with zeroth-, first- and  second-order mass-action kinetics satisfying Assumption \ref{hyp:1}. Assume further that one of the following equivalent statements holds:
  \begin{enumerate}[label=({\alph*})]
    \item There exists a $v\in\mathbb{R}^d_{>0}$ such that $v^TS_b=0$ for which the function $V(x)=v^Tx$ is a Foster-Lyapunov function for the corresponding Markov process.
        \item There exists a $v\in\mathbb{R}^d_{>0}$ such that the conditions
        \begin{equation}
    v^TS_uW_u<0\textnormal{ and } v^TS_b=0
  \end{equation}
  hold.
  \end{enumerate}
  Then, the reaction network is exponentially ergodic and the stationary distribution is light-tailed.
\end{theorem}
\begin{remark}
  If the condition $v^TS_b=0$ is relaxed to $v^TS_b\le0$ where at least one entry is nonzero, we cannot conclude anymore on the light-tailedness of the stationary distribution but the exponentially ergodicity property still holds.
\end{remark}

 %\subsubsection{Remarks}

Theorem \ref{thm:uni} and Theorem \ref{thm:bi} are interesting for different reasons. The first one is that they can be stated in a very accessible way using elementary linear algebra concepts. The second one is that the conditions can be numerically verified even for very large systems since it belongs to the tractable class of finite-dimensional linear programming problems for which many efficient algorithms exist; see e.g. \cite{Boyd:04}. It is important to mention that the number of constraints and variables scale linearly as a function of the number of species, not the number of reactions which can be typically much higher.%

\subsection{Antithetic integral control of stochastic reaction networks}

The antithetic integral controller has been introduced in \cite{Briat:15e} with the aim of developing an integral control theory for stochastic biochemical networks. Integral control is a cornerstone of control theory and control engineering as it allows to steer the output of a given system towards a desired constant set-point and to regulate this output around this value despite the presence of constant disturbances acting on the system. The idea behind the antithetic integral control is that it needs to be represented as a set of chemical reactions in order to be implemented in-vivo, such as in bacteria \cite{Aoki:19}. This controller takes the form of a (stochastic) reaction network $(\X{}\cup\Z{},\mathcal{R}_{\textnormal{AIC}})$ with species $\Z{}:=\{\Z{1},\Z{2}\}$ and reactions
\begin{equation}\label{eq:AIC}
  \Z{1}\rarrow{k}\Z{1}+\X{1}, \X{\ell}\rarrow{\theta}\X{\ell}+\Z{2},\phib\rarrow{\mu}\Z{1},\Z{2}+\Z{1}\rarrow{\eta}\phib
\end{equation}
where $\Z{1}$ and $\Z{2}$ are the actuating and the sensing species, respectively. The species $\X{\ell}$ is the measured/controlled species we would like to control by acting on the production rate of the actuated species $\X{1}$. The first reaction is referred to as the actuation reaction since it catalytically produces one molecule of the actuated species with a rate proportional to the actuating species. Symmetrically, the second reaction is the sensing reaction since it catalytically produces one molecule of the sensing species with a rate proportional to the measured species. The third reaction is the reference reaction since it sets part of the set-point for the stationary mean of the controlled species population. Finally, the last reaction is the comparison as it compares the populations of the controller species and acts as a nonlinear subtraction operator while, at the same time, closing the loop. Without this reaction the interconnected network would not be in closed loop since the populations of the controller species would be completely uncorrelated.

When the closed-loop network $(\X{}\cup\Z{},\mathcal{R}\cup\mathcal{R}_{\textnormal{AIC}})$ is ergodic, a quick inspection at the first-order moment dynamics allow us to state that
\begin{equation}
  \E[X_\ell(t)]\to\E_\pi[X_\ell]:=\mu/\theta\ \textnormal{as }t\to\infty
\end{equation}
where $\E_\pi$ denotes the expectation operator at stationarity. In this regard, the antithetic integral controller allows to steer to mean population of the controlled species to a desired set-point. It also able to achieve perfect adaptation for the closed-loop dynamics provided that the set-point is achievable; i.e. there must exist a positive steady-state for the closed-loop dynamics for which we have $\E_\pi[X_\ell]=\mu/\theta$.

We have the following result in the case of unimolecular networks with mass-action kinetics which is a slight extension of the one in \cite{Briat:15e}:
\begin{proposition}
  Let us consider the stochastic reaction network $(\X{},\mathcal{R})$ with first-order mass-action kinetics satisfying Assumption \ref{hyp:1}. Then, the following statements are equivalent:
  \begin{enumerate}[label=({\alph*})]
    \item there exist vectors $v\in\mathbb{R}^d_{>0}$ and $w\in\mathbb{R}^d_{\ge0}$, $w_1,w_\ell>0$,
  \begin{enumerate}[label=({\roman*})]
    \item $v^TS_uW_u<0$,
    \item $w^TS_uW_u+e_\ell^T=0$, and
    %\item $\dfrac{\mu}{\theta}>\dfrac{v^TS_0\lambda_0}{v_\ell}$
  \end{enumerate}
    \item The following conditions hold:
      \begin{enumerate}[label=({\roman*})]
    \item the matrix $S_uW_u$ is Hurwitz stable, and
    \item the system $(S_uW_u,e_1,e_\ell^T)$ is output controllable\footnote{i.e. $\rank\begin{bmatrix}
      e_\ell^Te_1 & e_\ell^TS_uW_ue_1 & \ldots & e_\ell^T(S_uW_u)^{d-1}e_1
    \end{bmatrix}=1$}.
      \end{enumerate}
    \item The closed-loop network $(\X{}\cup\Z{},\mathcal{R}_{\textnormal{AIC}})$  is ergodic and we have that \begin{equation}
        \E[X_\ell(t)]\to\mu/\theta\ \textnormal{as }t\to\infty
        \end{equation}
    \end{enumerate}
\end{proposition}
\begin{proof}
  The proof of the equivalence between the statement (a) and (b) has been obtained in \cite{Briat:15e} as well as the proof of the implication (a) $\Rightarrow$ (c). To prove the implication (c) $\Rightarrow(b)$ it is enough to remark that the stability and the output controllability of the first -order moment equation are necessary conditions allowing for the integral control of a system. Indeed, if the moment equation is unstable, then so will be the first-order moments of the closed-loop network and if the system is not output controllable, then one cannot change the value of output by suitably acting on the input. This completes the proof.
\end{proof}

As in Theorem \ref{thm:uni}, the above result connects algebraic conditions to the stability and the output-controllability of a linear time-invariant system, which are standard concepts of systems and control theory. The algebraic conditions can be solved using standard linear programming techniques. However, it is interesting to note that the conditions are independent of the gain $k$ and the annihilation parameter $\eta$ of the controller. This is quite surprising as it is well-known that setting the gain of an integral controller too high in the deterministic setting often results into the destabilization of the closed-loop dynamics (unless the system to be controlled satisfies certain strong properties).

%Interestingly, when there is no zeroth-order reactions, the above result can be reformulated in terms of control theoretic concepts:
%
%\begin{proposition}[\cite{Briat:15e} ]
%Let us consider the stochastic reaction network $(\X{},\mathcal{R})$ and assume that it only contains first-order reactions with mass-action kinetics. Let us further assume that the state-space of the underlying reaction network is irreducible and that the matrix $S_uW_u$ is invertible. Then, the following statements are equivalent:
%\begin{enumerate}[label=({\alph*})]
%  \item The following conditions hold:
%  \begin{enumerate}[label=({\roman*})]
%
%  \end{enumerate}
%  %
%  \item The closed-loop network $(\X{}\cup\Z{},\mathcal{R}\cup\mathcal{R}_{\textnormal{AIC}})$ is ergodic.
%\end{enumerate}
%As a result, we have that
%\begin{equation}
%  \E[X_\ell(t)]\to\mu/\theta\ \textnormal{as }t\to\infty.
%\end{equation}
%\end{proposition}

It is worth mentioning that the antithetic integral controller can be used to control more complex networks including bimolecular reactions or non-mass-action kinetics such as Michaelis-Menten kinetics \cite{Briat:15e}. However, the existing theoretical results do not yet cover those important classes of reaction networks. The difficulty arises from the potential loss of the (global) output-controllability condition. Simulation results tend to suggest that this controller should work properly even when the output-controllability property is not met globally, as is often the case for nonlinear systems.

\section{Phase-type distributed delays and their reaction network implementation}\label{sec:delays}

The aim of this section is to convince the reader that delays obeying phase-type distributions admit a natural reaction network representation and are relevant to consider. We first recall some theoretical basics on phase-type distributions and provide some examples in order to illustrate their richness in terms of behavior. In fact, those distributions are dense in the set of all probability distributions, which means that we can approximate arbitrarily closely any distribution, including heavy-tailed ones. We then provide a complete characterization of phase-type distributions in terms of simple algebraic conditions. Finally, under the existence assumption of a Markov process describing the considered phase-type distribution, we propose a simple procedure to construct the unique minimal unimolecular reaction network encoding this distribution. By unique and minimal, it is meant here that it is both the only network that exactly represents the considered distribution and the one that has the smallest number of reactions and molecular species.

\subsection{Preliminaries on phase-type distributions}

A phase-type distribution is a combination of mixtures and convolutions of exponential distributions. It is obtained by forming a system of interrelated Poisson processes (also known as phases) placed in series and parallel. It is represented by a random variable describing the time until absorption of a Markov process with one absorbing state starting from the initial condition $\alpha$ -- each state of the Markov process representing a phase of the overall process. The probability density function of the phase-type distribution $\PH(\alpha,H)$ is given by
\begin{equation}
  f(\tau)=\alpha e^{H\tau}H_0,\ \tau\ge0
\end{equation}
where $\alpha\in\mathbb{R}^{1\times m}_{\ge0}$, $||\alpha||_1=1$, $H\in\mathbb{R}^{m\times m}$ is a Hurwitz stable Metzler matrix and $H_0:=-H\mathds{1}_n$. The parameter $H$ is referred here to as the \emph{subgenerator matrix} and $\alpha$ as the \emph{input probability row vector}. Let $X\sim \PH(\alpha,H)$, then all the moments of this random variable are given by
\begin{equation}\label{eq:moment}
 \E[X^\ell]=(-1)^\ell\ell!\alpha H^{-\ell}\mathds{1}_m.
\end{equation}
%The cumulative distribution is given by
%\begin{equation}
%  F(\tau)=1-\alpha e^{H\tau}\mathds{1}_m,\ \tau\ge0.
%\end{equation}
%The Laplace transform of the distribution is given by
%\begin{equation}
%  \widehat{f}(s)=\alpha(sI-H)^{-1}H_0.
%\end{equation}
%
%%
The evolution of the probability distribution of the corresponding Markov process is described by the forward Kolmogorov equation
\begin{equation}
  \dot{p}(t)=p(t)\begin{bmatrix}
    0 & 0_{1\times m}\\
    H_0 & H
  \end{bmatrix}\ \textnormal{with}\ p(0)=\begin{bmatrix}
    0 & \alpha
  \end{bmatrix}.
\end{equation}
%Solving for this differential equation yields
%\begin{equation}
%  p(t)=\begin{bmatrix}
%    0 & \alpha
%  \end{bmatrix}\exp\left(\begin{bmatrix}
%    0 & 0_{1\times n}\\
%    H^0 & H
%  \end{bmatrix}t\right)=\begin{bmatrix}
%    1-\alpha e^{Ht}\mathds{1} & \alpha e^{Ht}
%  \end{bmatrix}
%\end{equation}
%where we can recognize the expression of the cumulative distribution.

\noindent\textbf{Hypoexponential and Erlang distributions.} Hypoexponential distributions consist of the convolution of a finite number of exponential distributions with possibly different various rates. When all the rates are equal, the hypoexponential distribution reduces to the Erlang distribution. Examples of Erlang distributions are given in Figure \ref{fig:erlang} where we can observe that this distribution is quite rich as it can take various forms. In the case of a hypoexponential distribution with four phases and four parameters  $\lambda_1,\ldots,\lambda_4>0$, the matrix $H$ is given by
\begin{equation}
  H=\begin{bmatrix}
    -\lambda_1 & \lambda_1 & 0 & 0\\
    0 & -\lambda_2 & \lambda_2 & 0\\
    0 & 0 & -\lambda_3 & \lambda_3\\
    0 & 0 & 0 & -\lambda_4
  \end{bmatrix}.
\end{equation}
Typical hypoexponential distributions are depicted in Figure \ref{fig:hypo} for randomly chosen parameters $\lambda_1,\ldots,\lambda_4>0$.\\
\begin{figure}
  \centering
  \includegraphics[width=0.8\textwidth]{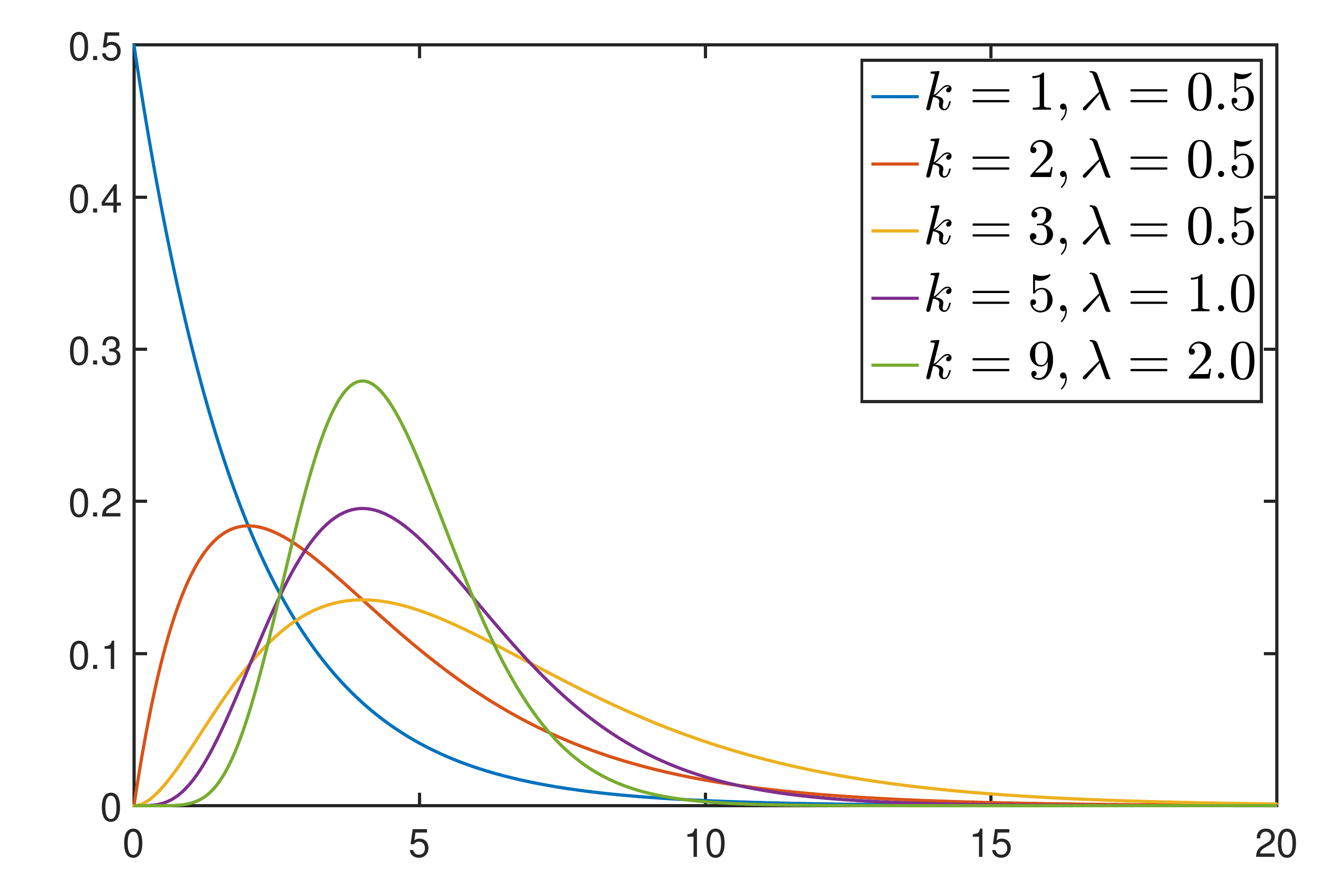}
  \caption{Examples of Erlang distributions}\label{fig:erlang}
\end{figure}
\begin{figure}
  \centering
  \includegraphics[width=0.8\textwidth]{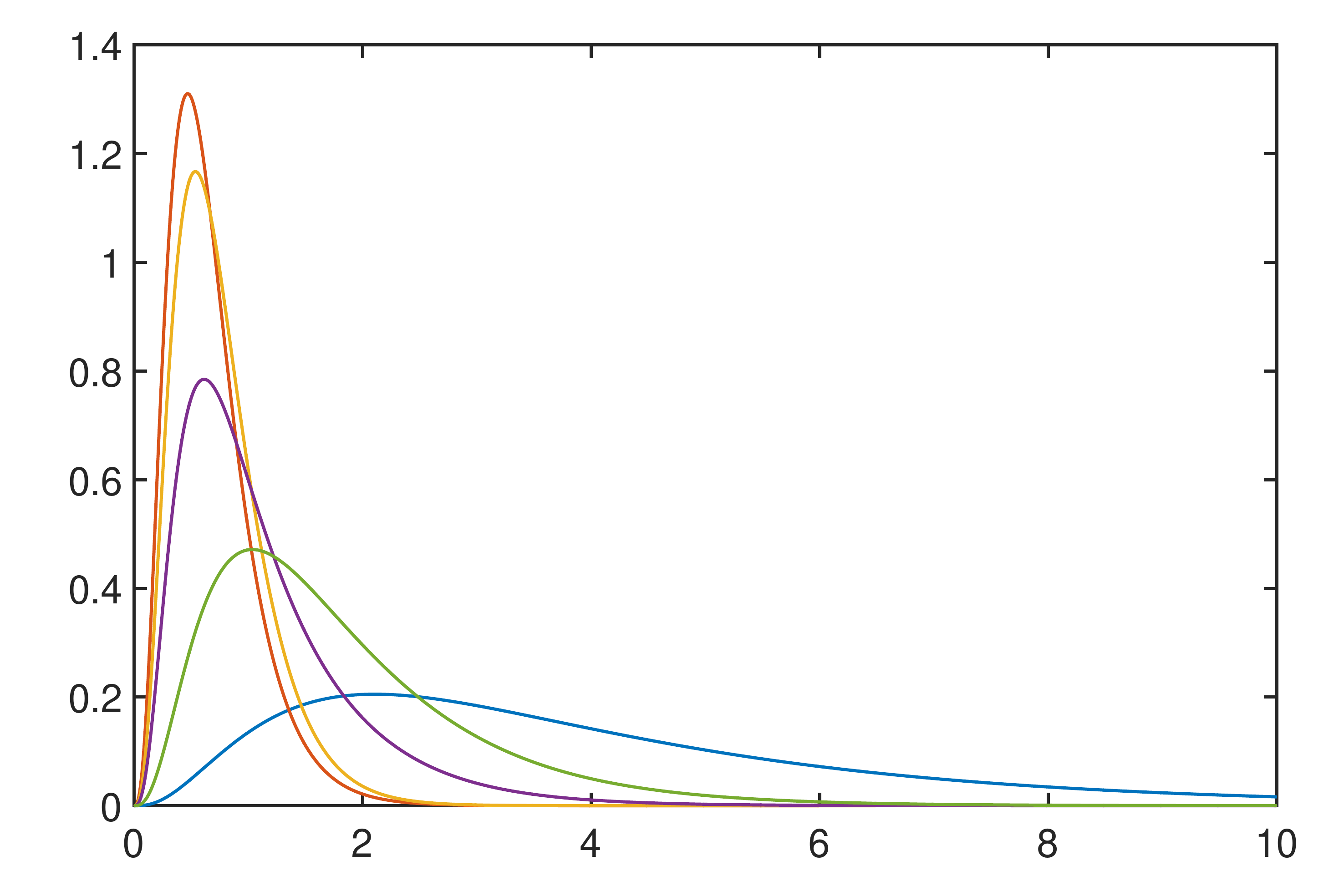}
  \caption{Examples of hypoexponential distributions}\label{fig:hypo}
\end{figure}

\noindent\textbf{Hyperexponential and Hyper-Erlang distributions.} Hyperexponential distributions consist of the mixture of exponential distributions, that is, the density function of a hyperexponential distribution is a convex combination of density functions of exponentially distributed random variables. Analogously, the Hyper-Erlang distribution has a density function consisting of a convex combination of density functions of Erlang random variables. For instance, let us consider two Erlang distributions. The parameter and the number of stages of the first one are 5 and 20, respectively. The second one has 5 and 80 as parameter and number of stages. The density function of the  considered Hyper-Erlang depicted in Figure \ref{fig:hypererlang} consists of the average of the density functions associated with the aforementioned Erlang distributions.

\begin{figure}
  \centering
  \includegraphics[width=0.8\textwidth]{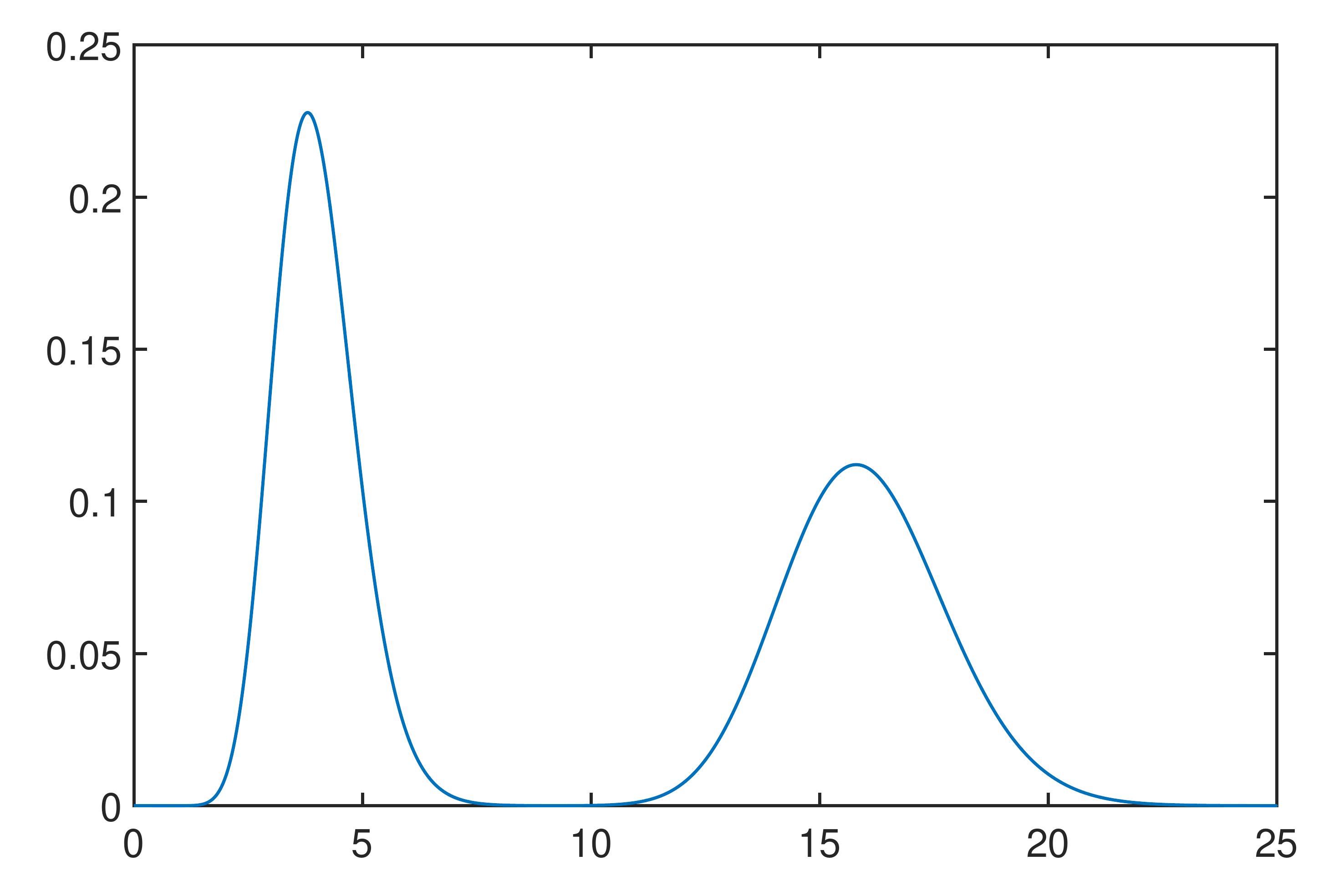}
  \caption{Example of a bimodal Hyper-Erlang distribution.}\label{fig:hypererlang}
\end{figure}

\subsection{Approximation of a constant delay}

In the deterministic setting, constant delays can be approximated arbitrarily closely by sequence of filters such as low-pass filters \cite{Marshall:79,Gawthrop:85} or all-pass filters such as lattice networks or Pad\'e approximants \cite{Tomlinson:63,Marshall:79,Lam:90}. It is known that the constant delay operator $\nabla_{\bar\tau}$ with delay $\bar \tau>0$ having $e^{-\bar\tau s}$ as transfer function can be approximated arbitrarily well by a series of $N$ low-pass filters with overall transfer function given by
\begin{equation}
  H_N(s):=\left(1+\dfrac{\bar\tau s}{N}\right)^{-N}, N\in\mathbb{Z}_{>0}.
\end{equation}
Note, moreover, that $e^{-\bar\tau s}$ is the Laplace transform of the Dirac distribution $\delta(t-\bar{\tau})$ centered around $\bar \tau$. We have the following standard result which is recalled for completeness:
\begin{proposition}[\cite{Marshall:79,Gawthrop:85}]
  We have that $H_N(s)\to e^{-s\bar\tau}$, for all $s\in\mathbb{C}$, as $N\to\infty$.
\end{proposition}
\begin{proof}
  Clearly, $H_N(s)$ can be rewritten as $H_N(s)=\exp(\log(H_N(s)))$. Hence,
\begin{equation}
  \begin{array}{rcl}
    H_N(s)    &=& \exp(\log(H_N(s)))\\
                    &=&  \exp(-N\log(1+\bar{\tau} s/N))\\
                    &\simeq& \exp(-\bar{\tau} s)
  \end{array}
\end{equation}
where we have used the principal branch of the logarithm and the fact that $\log(1+\bar{\tau}s/N)\approx \bar{\tau}s/N$ for a sufficient large $N$. As a result, we have that $H_N(s)\to e^{-s\bar{\tau}}$, for all $s\in\mathbb{C}$, as $N\to\infty$. This completes the proof.
\end{proof}

Interestingly, a similar result exists in the stochastic setting. Indeed, a constant delay $\bar\tau$ can be expressed as the limit (in distribution) of an Erlang random variable. This is stated in the following result:
\begin{proposition}\label{prop:detdel}
Let us then consider a random variable $\tau_N$ following an Erlang distribution with shape $N$ (i.e. number of phases) and rate $N/\bar{\tau}$. The corresponding density function is given by
\begin{equation}
  f_N(\xi)=\dfrac{N^N\xi^{N-1}e^{-N\xi/\bar{\tau}}}{\bar{\tau}^N(N-1)!}.
\end{equation}
  Then, we have that $\tau_N\to\bar{\tau}$ in distribution as $N\to\infty$.
\end{proposition}
\begin{proof}
Instead of proving that the cumulative distribution $F_N(\xi)$ of $\tau_N$ converges to the shifted Heaviside function $H(\xi-\bar{\tau})$ at all continuity points of $H(\xi-\bar{\tau})$ as $N\to\infty$, we propose a simpler alternative approach based on the mean and variance of the random variable $\tau_N$. From the moments expressions in \eqref{eq:moment}, it is immediate to see that the mean of $\tau_N$ is given by
\begin{equation}
  \E[\tau_N]=\bar{\tau}
\end{equation}
  and its variance by
  \begin{equation}
  V(\tau_N)=\bar{\tau}^2/N.
\end{equation}
Hence, we have that $\E[\tau_N]\to\bar \tau$ and $V(\tau_N)\to0$ as $N\to\infty$. Using the fact that the Erlang distributions are unimodal, then this implies that the density function $f_N(\xi)$ converges to the shifted Dirac delta $\delta(\xi-\bar\tau)$, and hence that $F_N(\xi)$ of $\tau_N$ converge to the shifted Heaviside function $H(\xi-\bar{\tau})$ as $N\to\infty$. This proves the result.
\end{proof}

The above result shows that, in the limit, the Erlang distribution tends to a Dirac with mass localized at $\bar{\tau}$. Figure \ref{fig:Dirac} depicts different Erlang distribution where we can see that as $N$ increases the shape of the Erlang distribution gets closer to the Dirac distribution.
\begin{figure}
  \centering
  \includegraphics[width=0.8\textwidth]{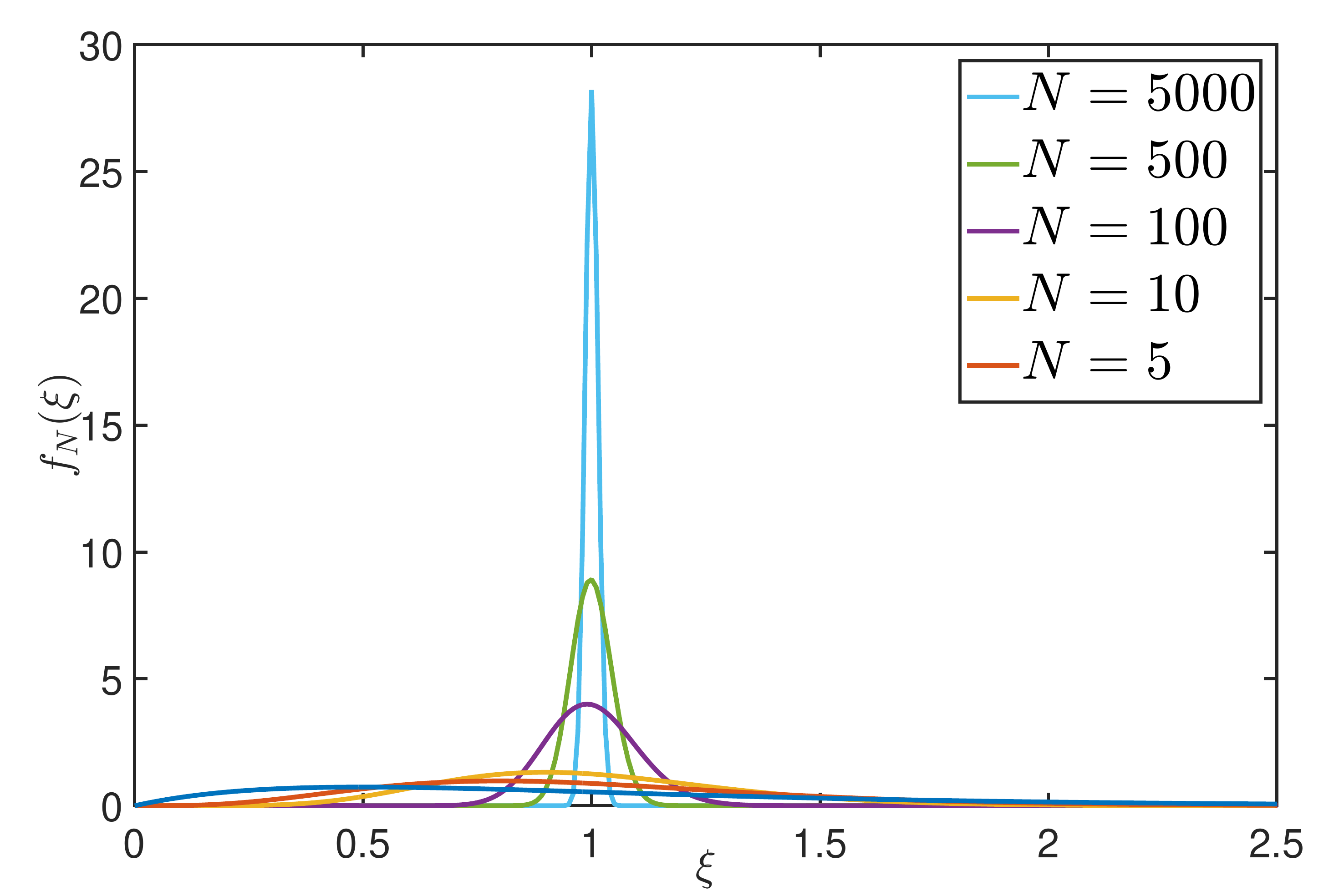}
  \caption{Erlang distributions with rate $N/\bar\tau$ and shape $N$ approximating the Dirac distribution.}\label{fig:Dirac}
\end{figure}

\subsection{Reaction network implementation of phase-type delays}\label{sec:implementation}

We now address two problems. The first one is, starting from a Hurwitz stable Metzler matrix $H\in\mathbb{R}^{m\times m}$,
can this matrix be used to represent a phase-type distribution? And, if so, how can we construct a reaction network that represents this distribution? Additional properties of the reaction network are also studied.\\

%and a probability row vector $\alpha\in\mathbb{R}^{1\times n}_{\ge0}$. We address, in this section, the following questions. The first one is: under what conditions the matrix $H$ can generate a transition matrix of a Markov process. The second question that will be addressed next is, given a suitable matrix $H$ that generates a transition matrix of  a Markov process, it is always possible to build a conservative (that preserves mass) reaction network corresponding to it. Finally, we will study several properties of the reaction network.

\subsubsection{Existence of a Markov process} The following result fully characterizes whether a given matrix $H\in\mathbb{R}^{m\times m}$ can be used to model a delay which phase-type distribution:
\begin{proposition}\label{prop:existence}
  Let us consider a matrix $H\in\mathbb{R}^{m\times m}$.
  \begin{enumerate}[label=({\alph*})]
    \item The matrix $H\in\mathbb{R}^{m\times m}$ is a subgenerator matrix.
    \item The following conditions hold:
     \begin{enumerate}[label=({\roman*})]
     \item it is Metzler and Hurwitz stable,
    \item $H\mathds{1}_m\le0$, and
    \item at least one entry of $H\mathds{1}_n$ is nonzero; i.e. $H\mathds{1}_m\ne0$.
  \end{enumerate}
  \end{enumerate}

\end{proposition}
\begin{proof}
  The matrix $H$ is a subgenerator matrix if and only if the matrix
  \begin{equation}
    \bar{H}:=\begin{bmatrix}
      0 & 0\\
      -H\mathds{1} & H
    \end{bmatrix}
  \end{equation}
  is a $Q$-matrix (i.e. the matrix $\bar H$ is Metzler and $\bar{H}\mathds{1}_m=0$) and that the first state is absorbing. Clearly, a necessary and sufficient condition for $\bar H$ to be Metzler is that $H$ be Metzler such that $H\mathds{1}\le0$. The last condition is obtained because the first state needs to be absorbing. This is equivalent to saying that $H\mathds{1}_n\ne0$. This proves the result.
\end{proof}

\subsubsection{Construction of the reaction network} We now address the problem of constructing a stochastic reaction network implementing a given phase-type distributed delay. Since the matrix $H$ is of dimension $m$, the network needs to have at least $m$ phases and, hence, at least $m$ molecular species. In fact, we will prove that we need exactly $m$ molecular species. We denote those species by $\D{1},\ldots,\D{m}$. In addition, the reaction network must satisfy the following properties:
\begin{itemize}
  \item the delay line must be conservative in the sense that nothing is lost or created inside the queue;
  \item the state-space of the reaction network describing the delay line needs to be irreducible.
\end{itemize}

The first constraint is easily met by exclusively considering conversion reactions. Catalytic or degradation reactions cannot be used as they do not preserve mass. Multimolecular reactions cannot be used since their propensity is nonlinear, while we need linear propensities to appropriately represent the matrix $H$. The second constraint may seem contradictory with the fact that the Markov process describing the delay distribution has an absorbing state. Recall that the delay is the time-to-absorption of this Markov process. However, what we are requiring here is that the state-space $\mathbb{Z}^m_{\ge0}$ of the stochastic reaction network be irreducible. The motivation behind this constraint is that we would like to obtain conditions for the ergodicity of delayed stochastic reaction networks. In this regard, adding delay lines should not cause of a loss of irreducibility.

The following result describes the construction of the unique minimal stochastic reaction network implementing a given phase-type distributed delay:
\begin{proposition}
  Let the matrix $H=[h_{ij}]_{i,j=1,\ldots,m}$ be a subgenerator matrix and $\alpha\in\mathbb{R}^{1\times n}$ be an input probability row vector. Then, the delay $\tau\sim \PH(\alpha,H)$ is exactly represented by the stochastic reaction network
  \begin{equation}
    \begin{array}{lclclclclcl}
      \phib&\rarrow{\alpha e_i}&\D{i},&& \D{i}&\rarrow{h_{ij}}&\D{j}, &&\D{i}&\rarrow{-e_i^TH\mathds{1}}&\phib,
    \end{array}
  \end{equation}
  where $i,j=1,\ldots,m$, $i\ne j$. Moreover, this network is minimal (minimal number of species and reactions) and its state-space is irreducible.
\end{proposition}
\begin{proof}
The queue (consisting of the reactions involving delay species as both reactants and products) is conservative since it only consists of conversion reactions. From the structure of the matrix $H$, the irreducibility is also immediate since we can see that any state can be reached from any other state through a sequence of reactions having a positive propensity. The minimality of the species comes from the fact that each species correspond to one phase. Since, we have $m$ phases, we need at least $m$ species. Finally, the minimality of the reactions comes from the fact that we have exactly one reaction for every nonzero off-diagonal entry of the matrix $H$. In this regard, every reaction is needed and one cannot choose a smaller number. This proves the minimality of the  number of reactions and, as a consequence, this network is unique.
\end{proof}

The above result shows that any phase-type distributed delay can be represented as a simple unimolecular chemical reaction network. Each produced molecule by the birth reactions gets destroyed by the death reactions after a delay $\tau\sim\textnormal{PH}(H,\alpha)$. This is very convenient  since it is  known from \cite{Briat:13i,Briat:16cdc,Briat:17ifacBio} that unimolecular stochastic reaction networks are usually well-behaved and easy to analyze using standard linear algebra tools.

\begin{example}
  Let us consider a delay that follows the Erlang distribution with rate $\lambda>0$ and shape $m$. In this case, we have that
\begin{equation}
  H=\begin{bmatrix}
    -\lambda & \lambda & 0 & \ldots & 0\\
    0 & -\lambda & \lambda & \ldots & 0\\
    \vdots &\ddots & \ddots & \ddots & \vdots\\
    \vdots & &\ddots& \ddots&\vdots\\
    0 & 0 & \ldots & 0 &-\lambda
  \end{bmatrix},H_0=\begin{bmatrix}
    0\\
    0\\
    \vdots\\
    \vdots\\
    \lambda
  \end{bmatrix},\alpha=\begin{bmatrix}
    1\\
    0\\
    \vdots\\
    \vdots\\
    0
  \end{bmatrix}^T.
\end{equation}
Elementary matrix calculations show that the probability density function of the time to absorption $\alpha e^{H\xi}H_0$ is equal to
\begin{equation}
  \dfrac{\lambda^N\xi^{N-1}e^{-\lambda\xi}}{(N-1)!},
\end{equation}
which confirms that $\textnormal{PH}(H,\alpha)$ generates an Erlang distribution with shape $N$ and rate $\lambda$. The corresponding reaction network is given by
\begin{equation}
    \begin{array}{lclclclclcl}
      \phib&\rarrow{1}&\D{1},&& \D{i}&\rarrow{\lambda}&\D{i+1},&& \D{m}&\rarrow{\lambda}&\phib,i=1,\ldots,m-1.
    \end{array}
\end{equation}
\end{example}

\section{Stochastic reaction networks with phase-type distributed delays}\label{sec:delayednetworks}

We know now that for any phase-type distributed delay corresponds a minimal unimolecular stochastic reaction network. We are then in position to define delayed reactions in reaction networks. Delayed reactions take the form
\begin{equation}\label{eq:delayedReaction}
 f_1(\X{})+f_2(\X{})\rarrow{k,\tau(H,\alpha)} f_1(\X{})+g(\X{})
\end{equation}
where $f_1(\X{})$ and $f_2(\X{})$ denote the preserved part of the reactants and the consumed part of the reactants, respectively. The term $g(\X{})$ represents the produced species/reactants. The first parameter of the reaction is the \emph{reaction rate}, assumed to be nonnegative, whereas the second one is the \emph{delay} where $\tau(H,\alpha)$ is a shorthand for $\tau\sim\PH(H,\alpha)$. We can then apply the construction procedure of Section \ref{sec:implementation} to get the delay line. However, we now interconnect the substitute the reactants of the birth reactions in the delay line by the reactant of the delayed reaction \eqref{eq:delayedReaction} and the products of the death reactions of the delay line by the products of the reaction \eqref{eq:delayedReaction}. However, attention should be paid as the reaction \eqref{eq:delayedReaction} can be interpreted in two different ways depending on its meaning. This is formalized below:
\begin{proposition}
  The delayed reaction \eqref{eq:delayedReaction} with the $m$-phase delay $\tau\sim\PH(H,\alpha)$ admits one of the following realizations:
 \begin{enumerate}[label=({\alph*})]
   \item Non-absorbing realization:
   \begin{equation}
   \begin{array}{rcl}
     f_1(\X{})+f_2(\X{})&\rarrow{k\alpha e_i}&f_1(\X{})+\D{i},\\
     \D{i}&\rarrow{h_{ij}}&\D{j},\\
     \D{i}&\rarrow{-e_i^TH\mathds{1}_m}&g(\X{})
   \end{array}
    \end{equation}
    for all $i=1,\ldots,m$, $i\ne j$.
    \item Absorbing realization:
    \begin{equation}
   \begin{array}{rcl}
   f_1(\X{})+f_2(\X{})&\rarrow{k\alpha e_i}&\D{i},\\
   \D{i}&\rarrow{h_{ij}}&\D{j},\\
\D{i}&\rarrow{-e_i^TH\mathds{1}_m}&f_1(\X{})+g(\X{})
   \end{array}
    \end{equation}
for all $i,j=1,\ldots,m$, $i\ne j$.
 \end{enumerate}
\end{proposition}
\begin{proof}
  It is immediate to see that we have two possible implementations for this delayed reaction. Either the preserved part of the reaction is non-absorbed in the queue or it is absorbed and restituted at the end of the queue.
\end{proof}

We justify the above result with the following examples:
\begin{example}
   Let us consider for instance the catalytic reaction
\begin{equation}
\X{1}+\X{2}\rarrow{k,\tau(H,\alpha)}\X{1}+\X{2}+\X{3}
\end{equation}
where $\X{1}$, $\X{2}$ and $\X{3}$ represent mRNA, ribosomes and protein molecules, respectively. Let us assume that the delay represents the protein maturation after translation, the letter being assumed here to happen instantaneously. In this regard, the mRNA and ribosome molecules should not be absorbed in the queue and we have
\begin{equation}
\begin{array}{rcl}
     \X{1}+\X{2}& \rarrow{k\alpha e_i} &\X{1}+\X{2}+\D{i},\\
     \D{i} &\rarrow{h_{ij}}& \D{j},\\
     \D{i} &\rarrow{-e_i^TH\mathds{1}_m} &\X{3}
\end{array}
\end{equation}
where $,i,j=1,\ldots,m,i\ne j$.
\end{example}

\begin{example}
Let us consider the same reaction as in the previous example but assume now that the delay represents the translation delay. In this regard, the ribosome and mRNA molecules should be absorbed in the queues since these molecules are made unavailable when the translation reaction takes place. Therefore, the queue should be made absorbing as
\begin{equation}
\begin{array}{rcl}
     \X{1}+\X{2}& \rarrow{k\alpha e_i} &\D{i},\\
     \D{i} &\rarrow{h_{ij}}& \D{j},\\
     \D{i} &\rarrow{-e_i^TH\mathds{1}_m} &\X{1}+\X{2}+\X{3}
\end{array}
\end{equation}
where $,i,j=1,\ldots,m,i\ne j$.
\end{example}

The above implementations may not be completely equivalent from the analysis viewpoint. Indeed, the first one may be more difficult to consider because of the catalytic nature of the bimolecular reaction which only produces mass. The second implementation may be simpler to consider as it involves a conversion reaction that both destroys and produces mass.

We are now in position to define the networks of interest:
\begin{define}[Delayed reaction network]
  A delayed reaction network is defined as the triplet $(\X{},\mathcal{R},\tau)$ where $\X{}$ is the set of molecular species, $\mathcal{R}$ is the set of reactions and $\tau$ is the vector of delays.
\end{define}
\begin{define}[Delay-free reaction network]
  Let us consider a delayed reaction network $(\X{},\mathcal{R},\tau)$. Its associated delay-free network is defined as $(\X{},\mathcal{R},0)=(\X{},\mathcal{R})$ and consists of setting all the delays to zero.
\end{define}
\begin{define}[Augmented reaction network]
  Let us consider a delayed reaction network $(\X{},\mathcal{R},\tau)$ and assume that the delays are phase-type distributed. Let us further decompose $\mathcal{R}$ as two disjoint sets $\mathcal{R}^0$ and $\mathcal{R}^\tau$ containing the non-delayed and the delayed reactions, respectively.  Then, its associated augmented network is defined as $(\X{}\cup\D{},\mathcal{R}^0\cup,\mathcal{R}_d)$ where $\D{}$ and $\mathcal{R}_d$  are the set of delay species and delay reactions, respectively.
\end{define}

\section{Ergodicity analysis of delayed reaction networks with phase-type distributed delays}\label{sec:ergodicity}

\subsection{State-space irreducibility}

%Whereas the delay-block is not irreducible since the states inside the block are all transient, this is not necessary the case of a network involving delays. Let us consider for instance the network
%%
%\begin{equation}
%  \phib\rarrow{k_1}\X{1},\ \X{1} \rarrow{\gamma_1}\phib,\ \X{1} \rarrow{k_2,\tau}\X{2} ,\ \X{2} \rarrow{\gamma_2}\phib.
%\end{equation}
%where $\tau$ obeys some phase-type distribution and such that the reaction-network delay implementation is irreducible when we remove the absorbing reaction. In this case, the overall network is irreducible since we can create as many $\X{1} $ as desired and feed the delay-network to reach any state before the molecules leave the delay-network.
The following result states the condition under which the delayed reaction network with phase-type distributed delays has an irreducible state-space:
\begin{proposition}
  Let us consider a delayed reaction network $(\X{},\mathcal{R},\tau)$ with phase-type distributed delays. Then, the following statements are equivalent:
  \begin{enumerate}[label=({\alph*})]
    \item The projected state-space onto $\mathbb{Z}^d_{\ge0}$ of the Markov process associated with the network  $(\X{},\mathcal{R},\tau)$ is irreducible.
    \item The state-space of the Markov process associated with the augmented network $(\X{}\cup\D{},\mathcal{R}^0\cup,\mathcal{R}_d)$ is irreducible.
    \item The state-space of the Markov process associated with the delay-free network $(\X{},\mathcal{R})$ is irreducible.
  \end{enumerate}
\end{proposition}
\begin{proof}
   Clearly, if the delayed network is irreducible then the delay-free must also be irreducible. To prove the converse, just observe that the delay-lines have an irreducible state-space by construction, hence the irreducibility only depends on the reaction network with the delays set to 0. This proves the equivalence between the two last statements. To prove the implication (c) $\Rightarrow (a)$, just note that delays are only temporal features and cannot change the irreducibility property of the state-space projected onto $\mathbb{Z}^d_{\ge0}$ provided that the delays are bounded with probability one, which is the case for phase-type distributed delays. The reverse implication is immediate here.
\end{proof}
%\begin{remark}
%  We consider here the projection of the state-space of  the Markov process associated with the delayed reaction network $(\X{},\mathcal{R},\tau)$ onto $\mathbb{Z}^d_{\ge0}$ since we are only interested in the irreducibility of the state-space associated with the current state of the system $X(t)$ and not the state-space associated with the history of the process $\{\X(s):s\le t\}$. Note that the state-space of
%\end{remark}

\subsection{Ergodicity of unimolecular networks with delays}\label{sec:uni_erg}

Let us consider here the delayed reaction network $(\X{},\mathcal{R},\tau)$ with $d$ species, $K$ reactions and $n$ delays. Each delay $\tau$ is phase-type distributed, i.e. $\tau_k\sim\PH(H_i,\alpha_i)$, $i=1,\ldots,n$, for some subgenerator matrix $H_i\in\mathbb{R}^{d_i\times d_i}$ and some input probability row vector $\alpha_i\in\mathbb{R}^{d_i}_{\ge0}$. We can then define the augmented network $(\X{}\cup\D{},\mathcal{R}^0\cup,\mathcal{R}_d)$. The stoichiometric matrix $S$ associated with this augmented reaction network is given by
\begin{equation}
  H=\begin{bmatrix}
    S_x & \vline & S_{\textnormal{in},x} &\vline &  S_{\textnormal{out},x} &\vline &  0\\
    \hline
    0 & \vline & S_{\textnormal{in},d} &\vline &  S_{\textnormal{out},d} &\vline &  S_d
  \end{bmatrix}
\end{equation}
where $S_x$ is the stoichiometric matrix associated with the reactions that do not involve any delay species and, conversely, $S_d$ is the stoichiometric matrix associated with reactions that only involve delay species. The stoichiometric matrices
\begin{equation}
  S_{\textnormal{in}} :=\begin{bmatrix}
  S_{\textnormal{in},x}\\S_{\textnormal{in},d}\end{bmatrix}\ \textnormal{and}\ S_{\textnormal{out}} :=\begin{bmatrix}
  S_{\textnormal{out},x}\\S_{\textnormal{out},d}

\end{bmatrix}
\end{equation}
corresponds to conversion reactions entering the delay lines and leaving the delay lines, respectively. The associated propensity functions are given by
\begin{equation}
  %\begin{array}{rcl}
    \lambda_x(x):=\begin{bmatrix}
      W_xx\\
      b_x
    \end{bmatrix}, \lambda_{\textnormal{in}}(x):=\begin{bmatrix}
      W_{\textnormal{in}}x\\
      b_{\textnormal{in}}
    \end{bmatrix},
    \lambda_{\textnormal{out}}(\delta):=W_{\textnormal{out}}\delta \textnormal{ and }
    \lambda_d(\delta):=W_d\delta
  %\end{array}
\end{equation}
where $(x,\delta)\in\mathbb{Z}_{\ge0}^d\times\mathbb{Z}_{\ge0}^{d_1+\ldots+d_n}$ is the state of the augmented network. The terms $W_xx$, $b_x$, $W_{\textnormal{in}}x$, $b_{\textnormal{in}}$, $W_{\textnormal{out}}\delta$ and $W_d\delta$ represent the propensity functions of the non-delayed first-order reactions, the non-delayed zeroth-order reactions, the first-order reactions entering the delay lines, the zeroth-order reactions entering the delay lines, the propensity of the reactions leaving the delay lines and the propensity of the reactions inside the delay lines, respectively.

Let us also define the following matrices
\begin{equation}\label{eq:ABCD1}
\begin{array}{rcl}
      A&:=&S_x\begin{bmatrix}
        W_x\\
        0
      \end{bmatrix}+S_{\textnormal{in},x}\begin{bmatrix}
        W_{\textnormal{in}}\\
        0
      \end{bmatrix}\\
      B&:=&S_{\textnormal{out},x}W_{\textnormal{out}}\\
      C&:=&S_{\textnormal{in},d}\begin{bmatrix}
        W_{\textnormal{in}}\\
        0
      \end{bmatrix}\\
      D&:=&S_{\textnormal{out},d}W_{\textnormal{out}}+S_dW_d
\end{array}
\end{equation}
Calculations show that
\begin{equation}\label{eq:ABCD2}
\begin{array}{rcl}
      W_{\textnormal{out}}&:=&-\diag_{i=1}^n(\mathds{1}^TH_i^T)\\
      S_{\textnormal{in},d}&=&\diag_{i=1}^n(\alpha_i^T)\\
      W_{\textnormal{in}}&=&\begin{bmatrix}
  q_1^T\\
  \vdots\\
  q_m^T
\end{bmatrix}\\
C&=&\begin{bmatrix}
  \alpha_1^Tq_1^T\\
  \vdots\\
  \alpha_mq_n^T\\
  \hline
  0\\
  \vdots\\
  0
\end{bmatrix}\\
      D&=&H^T=\diag_i(H_i)^T
\end{array}
\end{equation}
where $q_j=r_{j}e_{\sigma(j)}$ and $r_{j}$ is the reaction rate of the delayed reaction associated with the delay line $j$ and $e_{\sigma(j)}$ is the vector of zeros except at the entry $\sigma(j)$ where it is one. Finally, $\sigma(j)$ is the mapping $\sigma:\{1,\ldots,N\}\mapsto\{1,\ldots,d\}$ where $\sigma(j)=i$ if the species $\X{i}$ acts on the production of some delay species in the delay line $j$.

Interestingly, the stoichiometric matrix associated with the delay-free reaction network is given by
\begin{equation}
  S_{\textnormal{df}}:=\begin{bmatrix}
    S_x & \vline & S_{\textnormal{in},x}+S_{\textnormal{out},x}
  \end{bmatrix}
\end{equation}
with the propensity function
\begin{equation}
\lambda_{\textnormal{df}}(x):=\begin{bmatrix}
  W_xx\\
  b_x\\
  \hline
  W_{\textnormal{in}}x\\
  b_{\textnormal{in}}
\end{bmatrix}
\end{equation}
which leads to the characteristic matrix
\begin{equation}
  A_{\textnormal{df}}=S_x\begin{bmatrix}
    W_x\\0
  \end{bmatrix}+(S_{\textnormal{in},x}+S_{\textnormal{out},x})\begin{bmatrix}
    W_{\textnormal{in}}\\
    0
  \end{bmatrix}.
\end{equation}

The following proposition will be instrumental in proving the main result of the section:
\begin{proposition}\label{prop:Schurlike}
  We have $A_{\textnormal{df}}=A-BH^{-T}C$.
\end{proposition}
\begin{proof}
  Comparing the expressions for  $A_{\textnormal{df}}$ and $A$, we simply need to prove that
  \begin{equation}
      -W_{\textnormal{out}}H^{-T}C=\begin{bmatrix}W_{\textnormal{in}}\\0
    \end{bmatrix}.
  \end{equation}
  We know that $W_{\textnormal{in}}=\col_{i=1}^n(q_i^T)$. Now, we have
  \begin{equation}
    \begin{array}{rcl}
      -W_{\textnormal{out}}H^{-T}C       &=&     \diag_{i=1}^n(\mathds{1}^TH_i^T)H^{-T}\begin{bmatrix}
        \col_{i=1}^n(\alpha_i^Tq_i^T)\\
        \hline
        0
      \end{bmatrix}
                                                                               =\diag_{i=1}^n(\mathds{1}^T)\begin{bmatrix}
        \col_{i=1}^n(\alpha_i^Tq_i^T)\\
        \hline
        0
      \end{bmatrix}\\
                                                                               &=&\begin{bmatrix}
        \col_{i=1}^n(q_i^T)\\
        \hline
        0
      \end{bmatrix}
      =\begin{bmatrix}W_{\textnormal{in}}\\\hline0
    \end{bmatrix}.
    \end{array}
  \end{equation}
  This completes the proof.
\end{proof}

We can now state our main result on the ergodicity of delayed unimolecular reaction networks:
\begin{theorem}\label{th:delayed:uni}
  Let us consider a delayed reaction network $(\X{},\mathcal{R},\tau)$ with zeroth- and firth-order mass-action kinetics and phase-type distributed delays for which the delay-free counterpart $(\X{},\mathcal{R})$ satisfies Assumption \ref{hyp:1}. Then, the following statements are equivalent:
  \begin{enumerate}[label=({\alph*})]
   \item The delayed reaction network  $(\X{},\mathcal{R},\tau)$ is exponentially ergodic.
   \item The augmented reaction network  $(\X{}\cup\D{},\mathcal{R}^0\cup,\mathcal{R}_d)$ is exponentially ergodic.
    \item The matrix $$\mathcal{A}:=\begin{bmatrix}
        A & \vline & B\\
        \hline
        C & \vline & H^T
  \end{bmatrix}$$ is Hurwitz stable.
  \item The matrix $A_{\textnormal{df}}$ is Hurwitz stable.
  \item The delay-free network $(\X{},\mathcal{R})$ is exponentially ergodic.
  \end{enumerate}
\end{theorem}
\begin{proof}
By assumption, the state-space of the Markov process associated with the delay-free reaction network is irreducible, Equivalently, the state-spaces of the Markov processes associated with the augmented and the delayed reaction networks are also irreducible. Since, the delay-free as no closed components. This is equivalent to saying that the augmented and the delayed reaction networks do not have closed components as well.

The statements (a) and (b) are equivalent due to the equivalence of representations.

To prove the equivalence between (b) and (c), let $\tilde{\mathbb{A}}$ be the generator of Markov process representing the augmented reaction network and define the linear form $\textstyle V(x,\delta):=v_x^Tx+v_{\delta}^T\delta$ defined for $v_x\in\mathbb{R}^d_{>0}$ and $v_{\delta}\in\mathbb{R}^{d_1+\ldots+d_n}_{>0}$. Then, the Markov process representing the augmented network is exponentially ergodic if and only if there exist $c_1\ge0$ and $c_2>0$ such that $\tilde{\mathbb{A}}V(x,\delta)\le c_1-c_2V(x,\delta)$ for all $(x,\delta)\in\mathbb{Z}_{\ge0}^{d}\times\mathbb{Z}_{\ge0}^{d_1+\ldots+d_n}$. This can be reformulated as
\begin{equation}
  \tilde{\mathbb{A}}\left(v_x^Tx+v_{\delta}^T\delta\right)=\begin{bmatrix}
    v_x\\
    \hline
    v_{\delta}
  \end{bmatrix}^T\left(\begin{bmatrix}
        A & \vline & B\\
        \hline
        C & \vline & H^T
  \end{bmatrix}\begin{bmatrix}
    x\\
    \hline
    \delta
  \end{bmatrix}+\begin{bmatrix}
    S_xb_x\\
    \hline
    S_{\textnormal{in},d}b_{\textnormal{in}}
  \end{bmatrix}\right).
\end{equation}
  So, there exists some $c_1\ge0$ and $c_2>0$ such that $\textstyle\tilde{\mathbb{A}}V(x,\delta)\le c_1-c_2V(x,\delta)$ for all $(x,\delta)\in\mathbb{Z}_{\ge0}^{d}\times\mathbb{Z}_{\ge0}^{d_1+\ldots+d_n}$ if and only if $\mathcal{A}$ is Hurwitz stable. This proves the equivalence between the statements (b) and (c).

To prove the equivalence between the statements (c), (d) and (e), we first need the following result:
  \begin{lemma}
    Let us consider a Metzler matrix $M=[M_{ij}]_{i,j=1}^2$ where both $M_{11}$ and $M_{22}$ are square. Then, the matrix $M$ is Hurwitz stable if and only if the matrices $M_{11}-M_{12}M_{22}^{-1}M_{21}$ and $M_{22}$ are Hurwitz stable.
  \end{lemma}

Since $H$ is Hurwitz stable by construction, we can apply the above result with $M_{11}=A$, $M_{12}=B$, $M_{21}=C$ and $M_{22}=H^T$. As a result we get that $\mathcal{A}$ is Hurwitz stable if and only if $A-BH^{-T}C$ is also Hurwitz stable. From Proposition \ref{prop:Schurlike}, this is equivalent to say that $A_{\textnormal{df}}$ is Hurwitz stable and that the delay-free network is exponentially ergodic. This proves the result.
\end{proof}

The above result is interesting for multiple reasons. The first one is that the shape of the distribution or, equivalently, its complexity does not have any impact of the ergodicity of the network as long as it is unimolecular. A conclusion is that one can consider \emph{almost constant deterministic delays} as one can choose an arbitrarily large, yet finite, shape value in the Erlang distribution. Yet, one cannot conclude on the ergodicity of the network in the case of a constant deterministic delay.

\subsection{Ergodicity of bimolecular networks with delays}

We address here the case of bimolecular networks with delays and we will distinguish two cases. The first one is when only unimolecular networks are delayed whereas the second one is when the network is allowed to have delayed bimolecular reactions.

\subsubsection{No bimolecular reactions enters the queue}

We consider here a delayed reaction network $(\X{},\mathcal{R},\tau)$ with zeroth-, first- and second-order mass-action kinetics and  phase-type distributed delays. We can define $S_b$ to be the stoichiometric matrix of the augmented network associated with the bimolecular reactions as
\begin{equation}
  S_b=\begin{bmatrix}
    S_{b,x}\\
    \hline
    0
  \end{bmatrix}.
\end{equation}
The rest of the matrices are the same as for unimolecular networks; see Section \ref{sec:uni_erg}. We then have the following result:
\begin{theorem}
  Let us consider a delayed reaction network $(\X{},\mathcal{R},\tau)$ with zeroth-, first- and second-order mass-action kinetics and phase-type distributed delays for which the delay-free counterpart $(\X{},\mathcal{R})$ satisfies Assumption \ref{hyp:1}. Then, the following statements are equivalent:
\begin{enumerate}[label=({\alph*})]
   \item There exist vectors $v_x\in\mathbb{R}_{>0}^d$ and $v_{\delta}\in\mathbb{R}_{>0}^{d_1+\ldots+d_n}$ verifying $\begin{bmatrix}
  v_x^T & \vline & v_\delta^T
\end{bmatrix}S_{b}=0$ for which the function $V(x,\delta)=v_x^Tx+v_\delta^T\delta$ is a Foster-Lyapunov function for the Markov process.
    \item  There exist vectors $v_x\in\mathbb{R}_{>0}^d$ and $v_{\delta}\in\mathbb{R}_{>0}^{d_1+\ldots+d_n}$  such that the conditions
  \begin{equation*}
\begin{bmatrix}
    v_x\\
    \hline
    v_\delta
  \end{bmatrix}^T\begin{bmatrix}
        A & \vline & B\\
        \hline
        C & \vline & H^T
  \end{bmatrix}<0\ \textnormal{and }    \begin{bmatrix}
  v^T & \vline & v_\delta^T
\end{bmatrix}S_b=0
  \end{equation*}
  hold.
  \item There exist a vector $v\in\mathbb{R}_{>0}^d$ satisfying $v^TS_{b,x}=0$ and $v^TA_{\textnormal{df}}<0$.
\end{enumerate}
Moreover, when one of above statements holds, the Markov process describing the delayed reaction network $(\X{},\mathcal{R})$ is exponentially ergodic and the stationary distribution is light-tailed.
\end{theorem}
\begin{proof}
The equivalence from the two first statements comes from a direct application of Theorem \ref{prop:meyn}. The equivalence with the last one follows from the same arguments as for Theorem \ref{th:delayed:uni}.
\end{proof}

\subsubsection{At least one bimolecular reactions enters the queue}

We consider here a delayed reaction network $(\X{},\mathcal{R},\tau)$ with zeroth-, first- and second-order mass-action kinetics and  phase-type distributed delays.  As the case of delayed unimolecular reactions can be brought back to non-delayed unimolecular reactions, we assume here that the delays only act on bimolecular reactions.

The stoichiometric matrix of the delayed-network associated with the non-delayed bimolecular reactions is denoted by $S_b$ whereas
\begin{equation}
  S_b^d=\begin{bmatrix}
  \zeta_{r,1}^d-\zeta_{\ell,1}^d & \ldots & \zeta_{r,K_d}^d-\zeta_{\ell,K_d}^d
\end{bmatrix}
\end{equation}
is the stoichiometric matrix associated with the delayed bimolecular reactions where $K_d$ is the number of delayed biomolecular reactions. After augmenting the network to incorporate the delay species and reactions, the latter becomes
\begin{equation}
S_b^i=  \begin{bmatrix}
    -\zeta_{\ell,1}^d\mathds{1}^T_{\varsigma_1} & \ldots & -\zeta_{\ell,K}^d\mathds{1}^T_{\varsigma_K}\\
    J_i^1\\
    & \ddots\\
    & & J_i^K
  \end{bmatrix}
\end{equation}
where $J_i^k$ is a matrix with entries equal to zero or one, $k=1,\ldots,K_d$, such that on each column there is one and only one entry equal to one. We also define the matrix
\begin{equation}
B_b=\begin{bmatrix}
  J^o_1 & \ldots & J^o_K
\end{bmatrix}
\end{equation}
where $J^o_k=\zeta_{r,k}^d\mathds{1}^T\Lambda^o_k$ and $\Lambda^o_k$ is a diagonal matrix with nonnegative entries representing the rate parameters of the reactions leaving the queue.

We then have the following result:
\begin{theorem}\label{th:bidelay}
 Let us consider a delayed reaction network $(\X{},\mathcal{R},\tau)$ with zeroth-, first- and second-order mass-action kinetics and  phase-type distributed delays.  for which the delay-free counterpart $(\X{},\mathcal{R})$ satisfies Assumption \ref{hyp:1}. Then, the following statements are equivalent:
\begin{enumerate}[label=(\alph*)]
\item There exist $v_x\in\mathbb{R}_{>0}^d$ and $v_{\delta}\in\mathbb{R}^{d_1+\ldots+d_n}$, verifying $v_x^T S_b=0$ and $\begin{bmatrix}
      v_x^T & v_\delta^T
    \end{bmatrix}S_b^i=0$ such that the function $V(x,\delta)=v_x^Tx+v_\delta^T\delta$ is a Foster-Lyapunov function for the associated Markov process.
 \item There exist $v_x\in\mathbb{R}_{>0}^d$ and $v_{\delta}\in\mathbb{R}^{d_1+\ldots+d_n}$, such that the conditions
\begin{equation}
  v^TS_b=0,\quad \begin{bmatrix}
    v_x^T & \vline & v_{\delta}^T
  \end{bmatrix}S_b^i=0
\end{equation}
and
  \begin{equation}
    \begin{bmatrix}
    v\\
    \hline
    v_{\delta}
    \end{bmatrix}\begin{bmatrix}
      A & \vline & B_b\\
      \hline
      0 & \vline & H^T
    \end{bmatrix}<0
    \end{equation}
    where $\textstyle H=\diag_{i=1}^{K_d}(H_i)$ and $A$ is the same as for unimolecular networks.
  \item There exists a $v\in\mathbb{R}^d_{>0}$ such that
  \begin{enumerate}[label=({\roman*})]
   \item $v^TA<0$,
    \item $v^TS_b=0$, and
    \item $v^TS_b^d<0$.
  \end{enumerate}
\end{enumerate}
Moreover, when one of the above statements hold, then the Markov process describing the augmented reaction network $(\X{}\cup\D{},\mathcal{R}^0\cup\mathcal{R}_d)$ is exponentially ergodic and its stationary distribution is light-tailed.
\end{theorem}
\begin{proof}
Similar calculations as in the previous proofs show that the conditions of the first statement can be equivalently reformulated as the conditions in statement (b). This proves the equivalence between the statements (a) and (b). The conditions of statement (b) can be reformulated as
\begin{equation}\label{eq:jdskljdlsdajdaksldjlaj}
  \begin{array}{ccc}
        v_x^TS_b&=&0\\
        -v_x^T\zeta_{\ell,1}^d\mathds{1}^T+v_{\delta,1}^T J_i^1&=&0\\
        \vdots & \vdots& \vdots\\
        -v_x^T\zeta_{\ell,K_d}^d\mathds{1}^T+v_{\delta,K_d}^T J_i^{K_d}&=&0\\
        v_x^TA&<&0\\
        v_x^TJ_1^0+v_{\delta,1}^TH_1^T&<&0\\
       \vdots& \vdots& \vdots\\
         v_x^TJ_{K_d}^0+v_{\delta,K_d}^TH_{K_d}^T&<&0,
  \end{array}
\end{equation}
where $v_\delta=:(v_{\delta,1},\ldots,v_{\delta,{K_d}})\in\mathbb{R}_{>0}^{d_1}\times\mathbb{R}_{>0}^{d_{K_d}}$. We can put aside the conditions $v_x^TA<0$ and $v_x^TS_b=0$ (two first conditions in the statement (c)) and consider each of the $v_{\delta,k}$'s separately as the corresponding conditions are uncoupled.  The idea is to turn the system of inequality and equality conditions into a system of equality conditions. To this aim, let us define a positive vector $q_k$ of appropriate dimensions and consider the equality $v_x^TJ_k^0+v_{\delta,k}^TH_k^T=-q_k^T$. Since the matrix $H_k$ is Metzler and Hurwitz stable by construction, then we have that $H_k^{-1}\le0$. Hence, we have that $v_x^TJ_k^0H_k^{-T}+v_{\delta,k}^T=-q_k^TH_k^{-T}$ or, equivalently, $v_{\delta,k}^T=-v_x^TJ_k^0H_k^{-T}-q_k^TH_k^{-T}$. Combining this with the other equality yields
\begin{equation}
  \left[v_x^TJ_k^0H_k^{-T}+q_k^TH_k^{-T}\right]J_i^k+v_x^T\zeta_{\ell,k}^d\mathds{1}_{\varsigma_k}=0
\end{equation}
We note that
\begin{equation}
\begin{array}{rcl}
  J_k^0H_k^{-T}J_i^k&=&\zeta_{\ell,k}^d\mathds{1}^T\Lambda_k^oH_k^{-T}J_i^k\\
                                            &=& -v_x^T\zeta_{r,k}^d\mathds{1}_{\varsigma_k}^T
\end{array}
\end{equation}
where we have used the fact that $\mathds{1}^T\Lambda_k^oH_k^{-T}J_i^k$ is the negative of the static-gain of the delay-line, which is equal to a vector of ones due to the property of conservativity of delays. This leads to the following equality
\begin{equation}\label{eq:dsl;dkldskdlsldksdlsklskldkslkdl}
  v_x^T(\zeta_{\ell,k}^d-\zeta_{r,k}^d)\mathds{1}_{\varsigma_k}=-q_k^TH_k^{-T}J_i^k,
\end{equation}
for $k=1,\ldots,K_d$. Since $q_k>0$, then this means that a necessary condition for the conditions in \eqref{eq:jdskljdlsdajdaksldjlaj} is that $v_x^T(\zeta_{r,k}-\zeta_{\ell,k})<0$. We know that $H_k^{-T}J_i^k$ is full column rank since $H_k$ is invertible and $J_i^k$ is full-column rank. In this regard for any given $v_x$, we can find a $q_k$ such that \eqref{eq:dsl;dkldskdlsldksdlsklskldkslkdl} holds. This proves the equivalence with the last statement.
\end{proof}

\begin{example}
  Let us consider the epidemiological network
  \begin{equation}
  \begin{array}{c}
        \phib\rarrow{k_1}\X{1}\rarrow{\gamma_1}\phib,\\
        \phib\rarrow{k_2}\X{2}\rarrow{\gamma_2}\phib,\\
        \phib\rarrow{k_3}\X{3}\rarrow{\gamma_3}\phib,\\
        \X{2}+\X{3}\rarrow{k_i,\tau_i} 2\X{2},\\
        \X{2}\rarrow{k_r,\tau_r}\X{3},\\
        \X{3}\rarrow{k_s,\tau_s}\X{1}
  \end{array}
    \end{equation}
    where $\X{1}$, $\X{2}$ and $\X{3}$ denote the susceptible, infectious and recovered people, respectively. The first three rows represent the inflow and outflow of those people in the system, followed by the contamination reaction, the recovery reaction and the susceptibility reaction. We assume that only the three last reactions are affected by some delays. The stoichiometric matrix restricted to bimolecular reactions is given by $\begin{bmatrix}
      -1 & 1 &0
    \end{bmatrix}^T$ and hence Theorem \ref{th:bidelay} demands that $-v_1+v_2<0$. We also have that $S_b=0$ and
    \begin{equation}
      A=A_{\textnormal{df}}=\begin{bmatrix}
        -\gamma_1 & 0 & k_s\\
         0& -\gamma_2-k_r & 0\\
        0 & k_r & -\gamma_3-k_s
      \end{bmatrix}
    \end{equation}
    where we have already removed the delayed unimolecular reactions and replaced them by their delay-free counterparts. A necessary condition for the feasibility of $v^TA_{\textnormal{df}}<0$ is that the matrix $A_{\textnormal{df}}$ be Hurwitz stable. One can see that this matrix is structurally stable in the sense that it is Hurwitz stable for all positive values of its parameters. Indeed, one has that $\mathds{1}^TA_{\textnormal{df}}<0$ for all positive values of the parameters. To prove that it the network is also structurally ergodic in the presence of the delays, we need to verify $v^TA_{\textnormal{df}}<0$ for some $v>0$ such that $-v_1+v_2<0$. Let $v=\begin{bmatrix}
      1+\eps & 1 & 1
    \end{bmatrix}^T$, then we get that $$v^T A_{\textnormal{df}}=\begin{bmatrix}
      -\gamma_1(1+\eps) & -\gamma_2 & \eps k_s-\gamma_3
    \end{bmatrix}$$ and, hence, picking any $0<\eps<\gamma_3/k_s$ proves that the delayed reaction network is ergodic. Even stronger, the delayed reaction network is structurally ergodic.
\end{example}

\section{Moments equation}\label{sec:moments}

The goal of this section is to describe the influence of phase-type distributed delays on the moment dynamics and their stationary values.

\subsection{Dynamics of the first-order moments}

For simplicity, let us consider here delayed reaction network $(\X{},\,mathcal{R},\tau)$ with zeroth- and first-order mass-action kinetics. Let us consider the associated augmented network $(\X{}\cup\D{},\mathcal{R}^0\cup\mathcal{R}_d)$. The moment dynamics of the augmented network can be generically written as
  \begin{equation}\label{eq:momentsystemdsds}
  \dfrac{d}{dt}\begin{bmatrix}
    \E[X(t)]\\
    \E[D(t)]
  \end{bmatrix}=\begin{bmatrix}
    A & B\\
    C & H^T
  \end{bmatrix}\begin{bmatrix}
    \E[X(t)]\\
    \E[D(t)]
  \end{bmatrix}+\begin{bmatrix}
   b_0\\
   b_d
  \end{bmatrix}
\end{equation}
where the matrices are defined as in Section \ref{sec:uni_erg} together with
\begin{equation}
  b_0:= S_x\begin{bmatrix}
    0\\
    \hline
    b_x
  \end{bmatrix}\ \textnormal{and } b_d= S_{\textnormal{in},d}\begin{bmatrix}
    0\\
    \hline
    b_{\textnormal{in}}
  \end{bmatrix}.
\end{equation}
We then have the following result:
\begin{proposition}
The dynamics of the moments first-order moments of the delayed network $(\X{},\mathcal{R},\tau)$  is described by
  \begin{equation}\label{eq:moments2}
\begin{array}{rcl}
  \dfrac{d}{dt}\E[X(t)]&=&A\E[X(t)]+\int_0^tS_{\textnormal{out},x}f(s)\begin{bmatrix}
     W_{\textnormal{in}}\E[X(t-s)]\\
    \hline
    0
  \end{bmatrix}ds+b_0\\
  &&+Be^{H^Tt}\E[D(0)]+S_{\textnormal{out},x}F(t)\begin{bmatrix}
    0\\
    \hline
    b_{\textnormal{in}}
  \end{bmatrix}
\end{array}
\end{equation}
with $\textstyle F(t):=\diag_{i=1}^n(F_i(t))$  and $\textstyle f(s):=\diag_{i=1}^n(f_i(s))$ where $F_i(t):=1-\alpha_ie^{H_it}\mathds{1}$ and $f_i(s):=-\alpha_ie^{H_is}H_i\mathds{1}$ are the cumulative distribution and the probability density function of the delay $\tau_i\sim\PH(H_i,\alpha_i)$.
\end{proposition}
\begin{proof}
Solving for $ \E[D(t)]$ in the moment system \eqref{eq:momentsystemdsds} yields
\begin{equation}
  \E[D(t)]=e^{H^Tt}\E[D(0)]+\int_0^te^{H^Ts}C\E[X(t-s)]ds+H^{-T}(I-e^{H^Tt}) b_d.
  \end{equation}
After substitution of the above expression in  \eqref{eq:momentsystemdsds}, we obtain
\begin{equation}\label{eq:moments}
\begin{array}{lcl}
  \dfrac{d}{dt}\E[X(t)]&=&A\E[X(t)]+b_0+Be^{H^Tt}\E[D(0)]-H^{-T}(I-e^{H^Tt})b_d\\
  &&+B\int_0^te^{H^Ts}C\E[X(t-s)]ds.
\end{array}
\end{equation}
We have that
\begin{equation}
\begin{array}{rcl}
    Be^{H^Ts}C&=&S_{\textnormal{out},x}\diag_{i=1}^n(-\mathds{1}^TH_i^Te^{H_i^Ts})\begin{bmatrix}
      \col_{i=1}^m(\alpha_i^Tq_i^T)\\
      \hline
      0
    \end{bmatrix}\\
    &=&S_{\textnormal{out},x}\begin{bmatrix}
      \col_{i=1}^m(-\mathds{1}^TH_i^Te^{H_i^Ts}\alpha_i^Tq_i^T)\\
      \hline
      0
    \end{bmatrix}\\
    &=&S_{\textnormal{out},x}\begin{bmatrix}
     \col_{i=1}^m(f_i(s)q_i^T)\\
      \hline
      0
    \end{bmatrix}\\
    &=&S_{\textnormal{out},x}f(s)\begin{bmatrix}
      W_{\textnormal{in}}\\
      \hline
      0
    \end{bmatrix}.
\end{array}
\end{equation}
where $\textstyle f(s)=\diag_{i=1}^n(f_i(s))$ and $f_i(s)=-\alpha_ie^{H_is}H_i\mathds{1}$ is the probability density function of $\tau_i\sim\PH(H_i,\alpha_i)$.  We also have that
\begin{equation}
  \begin{array}{rcl}
    -BH^{-T}(I-e^{H^Tt})b_d&=&S_{\textnormal{out},x}\diag_{i=1}^n(\mathds{1}^TH_i^T)H^{-T}(I-e^{H^Tt})\diag_{i=1}^n(\alpha_i^T)\begin{bmatrix}
      0\\
      \hline
       b_{\textnormal{in}}
    \end{bmatrix}\\
    &=&S_{\textnormal{out},x}\diag_{i=1}^n(\mathds{1}^T(I-e^{H_i^Tt}))\diag_{i=1}^n(\alpha_i^T)\begin{bmatrix}
      0\\
      \hline
       b_{\textnormal{in}}
       \end{bmatrix}\\
       &=&S_{\textnormal{out},x}\diag_{i=1}^n(\mathds{1}^T(I-e^{H_i^Tt})\alpha_i^T)\begin{bmatrix}
      0\\
      \hline
       b_{\textnormal{in}}
       \end{bmatrix}\\
       &=&S_{\textnormal{out},x}F(t)\begin{bmatrix}
      0\\
      \hline
       b_{\textnormal{in}}
       \end{bmatrix}
  \end{array}
\end{equation}
where $\textstyle F(t)=\diag_{i=1}^n(F_i(t))$ and $F_i(t)=1-\alpha_ie^{H_it}\mathds{1}$ is the cumulative distribution function of $\tau_i\sim\PH(H_i,\alpha_i)$. The result follows.
\end{proof}
From the above expression, we can see that the first-order moment dynamics involves convolutions with kernels corresponding to the probability density functions of the delays. In this regard, the presence of phase-type distributed delays in the stochastic dynamics corresponds to filtering terms in the mean dynamics. Interestingly, this connects very well with the fact that if we substitute the convolution kernels by the Dirac distribution, we obtain a standard delay system since
\begin{equation}
  \int_0^t\delta(s-\bar\tau)x(t-s)ds=x(t-\bar\tau),t\ge\bar\tau.
\end{equation}
Even if the result has been derived for unimolecular reaction networks, it readily generalizes to networks having more complex propensities at the expense of clarity.

\subsection{Invariance of the stationary first-order moments}

We have shown in the previous section that the dynamics of the first-order moments are affected by the presence of phase-type distributed delays. We have also proven that an augmented network is ergodic provided that its delay-free counterpart is. As a result, the first-order moments converge to a unique equilibrium point. We prove here that the stationary value for the first-order moments for the augmented network coincides with that of the delay-free and the delayed network.
\begin{proposition}
Let us consider here delayed reaction network $(\X{},\mathcal{R},\tau)$ with zeroth- and first-order mass-action kinetics, and define its associated augmented network as $(\X{}\cup\D{},\mathcal{R}^0\cup\mathcal{R}_d)$. Assume that the Markov process associated with the augmented network is ergodic.

Then, the stationary first-order moments of the augmented reaction network (and thus that of the delayed reaction network) coincide with the stationary first-order moments of the delay-free reaction network.
\end{proposition}
\begin{proof}
Since the Markov process associated with the augmented network is ergodic, then there exists a unique stationary distribution and the limit $\lim_{t\to\infty}\E[X(t)]=\E_\pi[X]$ holds regardless of the initial conditions. Furthermore, it solves
\begin{equation}\label{eq:stat}
\begin{array}{rcl}
0&=&\lim_{t\to\infty}\left(A\E[X(t)]+\int_0^tS_{\textnormal{out},x}f(s)\begin{bmatrix}
     W_{\textnormal{in}}\E[X(t-s)]\\
    \hline
    0
  \end{bmatrix}ds+b_0\right.\\
  &&\left.+Be^{H^Tt}\E[D(0)]+S_{\textnormal{out},x}F(s)\begin{bmatrix}
    0\\
    \hline
    b_{\textnormal{in}}
  \end{bmatrix}\right).
\end{array}
\end{equation}
This yields
\begin{equation}
\left(A+S_{\textnormal{out},x}\begin{bmatrix}
    W_{\textnormal{in}}\\
    \hline
  0
  \end{bmatrix}\right)\E_\pi[X]+S_x\begin{bmatrix}
    b_x\\
    \hline
    0
  \end{bmatrix}+S_{\textnormal{out},x}\begin{bmatrix}
    0\\
    \hline
    b_{\textnormal{in}}
  \end{bmatrix}=0
\end{equation}
where we have used the fact that $\textstyle \int_0^\infty f(s)ds=I$ and $F(t)\to0$ as $t\to\infty$. Finally, noting that the term in brackets is equal to $A_{\textnormal{df}}$, the result follows.
\end{proof}

This result then demonstrates that phase-type distributed delay do not affect the stationary value of the first-order moments.

\subsection{Variance is affected by the delays}

In the previous section, we have shown that the stationary first-order moments of the network remain the same regardless the presence of the delays. However, the variance is influenced by the presence of the delay. Unfortunately, it seems difficult to prove that in the general setting. So, instead we will prove that in the case of a gene expression network with maturation delay. To this aim, let us consider the gene expression network
\begin{equation}
\begin{array}{l}
    \phib\rarrow{k_1}\X{1}\rarrow{\gamma_1}\phib,\\
    \X{1}\rarrow{k_2}\X{1}+\X{2},\\
    \X{2}\rarrow{\gamma_2}\phib
\end{array}
\end{equation}
and its delayed counterpart
\begin{equation}
\begin{array}{l}
    \phib\rarrow{k_1}\X{1}\rarrow{\gamma_1}\phib,\\
    \X{1}\rarrow{k_2}\X{1}+\D{1},\\
    \X{2}\rarrow{\gamma_2}\phib,\\
    \D{1}\rarrow{\lambda}\X{2}\\
\end{array}
\end{equation}
where $\X{1}$ and $\X{2}$ denote the mRNA and the protein species, respectively. The delay only consist of one phase here and is therefore exponentially distributed with rate $\lambda$. Calculations show that the protein variance is
\begin{equation}
  V(X_2)=\mu\left(1+\dfrac{k_2}{\gamma_1+\gamma_2}\right)
\end{equation}
in the delay-free case but is equal to
\begin{equation}
  V_\lambda(X_2)=\mu\dfrac{\gamma_1\gamma_2(\gamma_1+\gamma_2)+ \lambda(\gamma_1+\gamma_2)(\gamma_1+\gamma_2+k_2) + \lambda^2(\gamma_1+\gamma_2+ k_2)}{(\gamma_1 + \gamma_2)(\gamma_1 + \lambda)(\gamma_2 + \lambda)}.
\end{equation}
Not surprisingly, we have that $V_\lambda(X_2)\to V(X_2)$ as $\lambda\to\infty$ since, in this case, there is no waiting time in the queue. To compare, we define the ratio
\begin{equation}
  R(\lambda)=\dfrac{V_\lambda(X_2)}{ V(X_2)}=\dfrac{\gamma_1\gamma_2(\gamma_1+\gamma_2)}{(\gamma_1 + \lambda)(\gamma_2 + \lambda)(\gamma_1+\gamma_2+k_2)}+ \dfrac{\lambda(\gamma_1+\gamma_2) + \lambda^2}{(\gamma_1 + \lambda)(\gamma_2 + \lambda)}.
\end{equation}
Calculations show that
 \begin{equation}
   R'(\lambda)=\dfrac{\gamma_1\gamma_2(\gamma_1k_2 + \gamma_2k_2 + 2k_2\lambda)}{(\gamma_1 + \gamma_2 + k_2)(\gamma_1+\lambda)^2(\gamma_2 + \lambda)^2}
 \end{equation}
and, therefore, the function $V_\lambda(X_2)$ is strictly increasing and verifies
\begin{equation}
  \mu<V_\lambda(X_2)<V(X_2)=\mu\left(1+\dfrac{k_2}{\gamma_1+\gamma_2}\right)
\end{equation}
for all $\lambda\in(0,\infty)$. This demonstrates that in the case of gene expression, the protein variance is influenced by the presence of the delay. Interestingly, one can see that the variance is smaller and that the queue has a filtering effect. Another interesting fact is the lower bound which imposes a lower bounds on how much the noise can be reduced and this lower bound corresponds to the variance of a Poisson process.

\section{Antithetic integral control of unimolecular reaction network with phase-type distributed delays}\label{sec:AIC}

We now address the problem of controlling a delayed stochastic reaction network $(\X{},\mathcal{R},\tau)$ with first-order mass action kinetics using a delayed antithetic integral controller
\begin{equation}\label{eq:delayedAIC}
  \Z{1}\rarrow{k,\tau(H_1,\alpha_1)}\Z{1}+\X{1}, \X{\ell}\rarrow{\theta,\tau(H_n,\alpha_n)}\X{\ell}+\Z{2},\phib\rarrow{\mu}\Z{1},\Z{2}+\Z{1}\rarrow{\eta}\phib.
\end{equation}
Note that the annihilation reaction does not need to be delayed since it is a degradation reaction. Indeed, adding a delay line here will not change the behavior of the system as the compound will simply follow the delay line and be degraded in the end. However, from the controller point of view, degradation already occurred when the compound entered the delay line. Similarly, the reference reaction does not need to be delayed since from the point of view of the controller, only the output of the delay line will be visible and the actual content of the delay line does not matter.

The closed-loop network obtained from the interconnection of $(\X{},\mathcal{R},\tau)$ and the antithetic integral controller \ref{eq:delayedAIC} is denoted by $(\X{}\cup\Z{},\mathcal{R}\cup\mathcal{R}_{\textnormal{AIC}},\tau)$. The associated augmented and delay-free networks are denoted by $(\X{}\cup\Z{}\cup\D{},\mathcal{R}^0\cup\mathcal{R}_d\cup\mathcal{R}^0_{\textnormal{AIC}}\cup\mathcal{R}_{\textnormal{AIC},d})$ and $(\X{}\cup\Z{},\mathcal{R}\cup\mathcal{R}_{\textnormal{AIC}},0)$, respectively.

We assume that the actuated species $\X{1}$ is produced after a stochastic time-varying delay modeled by the first delay-line whereas the sensing reaction is using the last delay line. In this regard, we have that $q_1=0$, $q_n=\theta e_\ell$. For technical reasons, we assume that those delays are such that $\alpha_1,\alpha_n,H_1\mathds{1}$ and $H_n\mathds{1}$ have one and only one nonzero entry. This assumption is only very mildly restrictive. Based on these facts, we can define the following matrices
\begin{equation}
  \begin{array}{rcl}
  B&:=&\begin{bmatrix}
    -e_1\mathds{1}^TH_1^T &\vline& S_{\textnormal{out},x}W_{\textnormal{out}} &\vline& 0
  \end{bmatrix}\\
C&=& \begin{bmatrix}
  0\\
  \hline
  \alpha_2^Tq_2\\
  \vdots\\
  \alpha_{n-1}^Tq_{n-1}\\
  \hline
  \alpha_n^Tq_n
\end{bmatrix}
  \end{array}
\end{equation}
where the first part corresponds the input delay, the middle part the delays intrinsic to the system and the last part of the measurement delay.

We have the following result:
\begin{theorem}
  Let us consider a delayed reaction network $(\X{},\mathcal{R},\tau)$ with only first-order mass-action kinetics and phase-type distributed delays. We assume that the associated delay-free network satisfies Assumption \ref{hyp:1}. Then, the following statements are equivalent:
  \begin{enumerate}[label=({\alph*})]
  \item  The augmented closed-loop stochastic reaction network $(\X{}\cup\Z{}\cup\D{},\mathcal{R}^0\cup\mathcal{R}_d\cup\mathcal{R}^0_{\textnormal{AIC}}\cup\mathcal{R}_{\textnormal{AIC},d})$ is ergodic and we have that $\E[X_\ell(t)]\to\mu/\theta$ as $t\to\infty$.
    \item There exists some vectors $(v_x,v_\delta)\in\mathbb{R}_{>0}^{d}\times\mathbb{R}^{d_1+\ldots+d_n}$ and $(w_x,w_\delta)\in\mathbb{R}_{\ge0}^{d}\times\mathbb{R}^{d_1+\ldots+d_n}$ such that the conditions
        \begin{equation}
          \begin{bmatrix}
            v_x\\v_\delta
          \end{bmatrix}^T\begin{bmatrix}
    A & \vline & B\\
    \hline
    C & \vline & H^T
  \end{bmatrix}<0,
  \end{equation}
  \begin{equation}
      \begin{bmatrix}
            w_x\\w_\delta
          \end{bmatrix}^T\begin{bmatrix}
    A & \vline & B\\
    \hline
    C & \vline & H^T
  \end{bmatrix}+\begin{bmatrix}
    0 & \vline & 0 & \ldots & -\mathds{1}^TH_n^T
  \end{bmatrix}=0
  \end{equation}
  and
  \begin{equation}
    \begin{bmatrix}
    0&
    \vline&
    \alpha_1&
    0&
    \ldots&
    0
  \end{bmatrix}w>0
  \end{equation}
        hold.
  \item There exists some vectors $v_x\in\mathbb{R}_{>0}^{d}$ and $w_x\in\mathbb{R}_{\ge0}^{d}$ satisfying $v_x^Te_1>0$, $w_x^Te_1>0$, $w_x^Te_\ell>0$, such that the conditions
  \begin{equation}
    v^TA_{\textnormal{df}}<0\textnormal{ and } w^TA_{\textnormal{df}}+e_\ell^T=0
  \end{equation}
  hold.
  \item The delay-free closed-loop stochastic reaction network $(\X{}\cup\Z{},\mathcal{R}\cup\mathcal{R}_{\textnormal{AIC}},0)$ is ergodic and we have that $\E[X_\ell(t)]\to\mu/\theta$ as $t\to\infty$.
  \end{enumerate}
\end{theorem}
\begin{proof}
  The conditions of the first statement are nothing else but the ergodicity and controllability conditions for the augmented reaction network. To prove the equivalence between the conditions, first note that we have that
\begin{equation}
  \begin{bmatrix}
            v_x\\v_\delta
          \end{bmatrix}^T\begin{bmatrix}
    A & B\\
    C & H^T
  \end{bmatrix}<0
\end{equation}
if and only if $v_x^T(A-BH^{-T}C)=v_x^TA_{\textnormal{df}}<0$. We also have that
\begin{equation}
  \begin{bmatrix}
            w_x\\w_\delta
          \end{bmatrix}^T\begin{bmatrix}
        A & B\\
      C & H^T
  \end{bmatrix}+\begin{bmatrix}
    0 & \vline & 0 & \ldots & -\mathds{1}^TH_n^T
  \end{bmatrix}=0
\end{equation}
together with
\begin{equation}
  \begin{bmatrix}
            w_x\\w_\delta
          \end{bmatrix}^T\begin{bmatrix}
    0\\
    \hline
    \alpha_1^T\\
    0\\
    \vdots\\
    0
  \end{bmatrix}>0.
\end{equation}
So, from $w_x^TB+w_\delta^TH^T=0$, we get that
\begin{equation}
  \begin{array}{rcl}
    w_\delta^T&^=&-(\begin{bmatrix}
    0 & \vline & 0 & \ldots & -\mathds{1}^TH_n^T
  \end{bmatrix}+w_x^TB)H^{-T}
  \end{array}
\end{equation}
and hence
\begin{equation}
 w_x^T(A-BH_n^{-T}C)+\begin{bmatrix}
    0 & \ldots & -\mathds{1}^TH_n^T
  \end{bmatrix}H^{-T}C=0
\end{equation}
which is equivalent to the condition that $w_x^T(A-BH^{-T}C)+e_\ell^T=0$. Using this value for $w_\delta$, we also get that
\begin{equation}
\begin{array}{rcl}
    w^T\begin{bmatrix}
    0\\
    \hline
    \alpha_1^T\\
    0\\
    \vdots\\
    0
  \end{bmatrix}&=&w_\delta^T\begin{bmatrix}
    \alpha_1^T\\
    0\\
    \vdots\\
    0
  \end{bmatrix}\
  =-(\begin{bmatrix}
    0 & \vline & 0 & \ldots & -\mathds{1}^TH_n^T
  \end{bmatrix}+w_x^TB)H^{-T}\begin{bmatrix}
    \alpha_1^T\\
    0\\
    \vdots\\
    0
  \end{bmatrix}\\
  &=& w_x^TBH^{-T}\begin{bmatrix}
    \alpha_1^T\\
    0\\
    \vdots\\
    0
  \end{bmatrix}\\
  &=&w_x^Te_1\mathds{1}^TH_1^TH_1^{-T}\alpha_1^T\\
  &=&w_x^Te_1
\end{array}
\end{equation}
where we have used the fact that $\mathds{1}^T\alpha_1^T=1$. We have proven the equivalence between the statements. However, it remains to prove that $\E[X_\ell(t)]\to\mu/\theta$. To do so, we first write down the moments equation
  \begin{equation}
  \begin{array}{rcl}
  \dfrac{d}{dt}\begin{bmatrix}
    \E[X(t)]\\
    \hline
    \E[D(t)]
  \end{bmatrix}  &=&\begin{bmatrix}
    A & \vline & B\\
    \hline
    C & \vline & H^T
  \end{bmatrix}\begin{bmatrix}
    \E[X(t)]\\
    \E[D(t)]
  \end{bmatrix}+\begin{bmatrix}
    0\\
    \hline
    k\alpha_1^T\\
    0\\
    \vdots\\
    0
  \end{bmatrix}\E[Z_1(t)]\\
  \dfrac{d}{dt}\E[Z_1(t)]&=&\mu-\eta\E[Z_1(t)]\E[Z_2(t)]\\
   \dfrac{d}{dt}\E[Z_2(t)]&=&\begin{bmatrix}
     0 & \vline & 0 & \ldots & 0 & -\theta\mathds{1}^TH_n^T
   \end{bmatrix}\begin{bmatrix}
    \E[X(t)]\\
    \E[D(t)]
  \end{bmatrix}-\eta\E[Z_1(t)]\E[Z_2(t)].
  \end{array}
  \end{equation}
The stationary solution is given by
\begin{equation}
  \begin{bmatrix}
    \E_\pi[X]\\
    \E_\pi[D]
  \end{bmatrix}=-k\begin{bmatrix}
    A & B\\
    C & H^T
  \end{bmatrix}^{-1}\begin{bmatrix}
    0\\
    \hline
    \alpha_1^Tk\E_\pi[Z_1]\\
    0\\
    \vdots\\
    0
  \end{bmatrix}
\end{equation}
and
\begin{equation}
  \mu/\theta=\begin{bmatrix}
    0 & \vline & 0 & \ldots & -\mathds{1}^TH_n^T
  \end{bmatrix}\begin{bmatrix}
    \E_\pi[X]\\
    \E_\pi[D]
  \end{bmatrix}
\end{equation}
where $\E_\pi$ denotes the expectation operator at stationarity. We use now the following formula for the inversion of block-matrices
\begin{equation}
\begin{array}{rcl}
  \begin{bmatrix}
    A & B\\
    C & H^T
  \end{bmatrix}^{-1}=\begin{bmatrix}
    (A-BH^{-T}C)^{-1} & (A-BH^{-T}C)^{-1}BH^{-T}\\
    -H^{-T}C(A-BH^{-T}C)^{-1} &  H^{-T}+H^{-T}C(A-BH^{-T}C)^{-1}BH^{-T}
  \end{bmatrix}
  %\\
%  &=&\begin{bmatrix}
%    \left(A+\sum_iS_{\textnormal{out},x}^iq_i^T\right)^{-1} & -\left(A+\sum_iS_{\textnormal{out},x}^iq_i^T\right)^{-1}
%  \end{bmatrix}
\end{array}
\end{equation}
where
\begin{equation}
  \begin{array}{rcl}
    A-BH^{-T}C&=&A_{\textnormal{df}}\\
    H^{-T}C&=&\begin{bmatrix}
      0\\
      H_2^{-T}\alpha_2^Tq_2^T\\
      \vdots\\
      H_n^{-T}\alpha_n^Tq_n^T
    \end{bmatrix}\\
    BH^{-T}&=&=\begin{bmatrix}
      -e_1\mathds{1}^TH_1^T & S_{\textnormal{out},x}W_{\textnormal{out}} & 0
    \end{bmatrix}.
  \end{array}
\end{equation}
Finally, we have that
\begin{equation}
  -\begin{bmatrix}
    0 & \vline & 0 & \ldots & -\mathds{1}^TH_n^T
  \end{bmatrix} \begin{bmatrix}
    A & B\\
    C & H^T
  \end{bmatrix}^{-1}\begin{bmatrix}
    0\\
    \hline
    \alpha_1^Tk\E_\pi[Z_1]\\
    0\\
    \vdots\\
    0
  \end{bmatrix}=\mu/\theta
\end{equation}
and, hence,
\begin{equation}
  \begin{bmatrix}
    0 & \ldots & 0 & \mathds{1}^TH_n^T
  \end{bmatrix}\left(H^{-T}+H^{-T}C(A-BH^{-T}C)^{-1}BH^{-T}\right)\begin{bmatrix}
    \alpha_1^Tk\E_\pi[Z_1]\\
    0\\
    \vdots\\
    0
  \end{bmatrix}=\mu/\theta.
\end{equation}
Using the explicit form for the matrices $A,B,C$ and $H$ yields $\E_\pi[Z_1]=\dfrac{\mu}{kg\theta}$ where $g:=e_\ell^T(A-BH^{-T}C)^{-1}e_1$ is the static-gain of the network. Note, moreover, that $g$ is the static-gain of the non-delayed network. Since we have that $\E_\pi[X_\ell]=gk\E_\pi[Z_1]$, then we obtain that
\begin{equation}
  \E_\pi[X_\ell]=\mu/\theta.
\end{equation}
This proves that $\X{\ell}$ is also regulated and completes the proof.
\end{proof}

As for the ergodicity analysis, we can see that the conditions also reduce to the condition for delay-free networks. In this regard, phase-type distributed delays are not harmful to the stability properties of the closed-loop network unlike in the deterministic setting.

\section{Concluding statements}\label{sec:conclusion}

Delays that are phase-type distributed have been shown to be natural to consider in the context of stochastic reaction networks. In fact, stochastic time-varying delays can be represented as reaction networks with conversion reactions. In this regard, any reaction network with delayed reactions having a delay that is phase-type distributed can be equivalently represented by a reaction network with augmented state-space. Yet, the dimension of the state-space remains finite. In this regard, existing results for the analysis and the control of stochastic reaction networks remain applicable and do not need to be extended to the infinite-dimensional case -- a difficult task. We first characterize all the delays that are phase-type distributed in terms of algebraic conditions. We then provide an explicit way for building the associated reaction network and interconnect it to the original network. Ergodicity tests are then provided and it proves that for unimolecular networks, the delayed reaction network is ergodic if and only if the delay-free network is ergodic as well. In this regard, delays are not harmful to the ergodicity property. This also indicate that delays can be arbitrarily complicated as long as they are phase-type distributed. For bimolecular networks, the situation is more complex but it can be shown that in certain cases the ergodicity conditions of the delayed network are fulfilled if and only if the conditions are fulfilled for the delay-free network. The analysis of the first moment equation demonstrate that the delays yield additional convolution terms in the mean dynamics of the molecular species. However, they do not change the stationary mean values and the stationary means are the same as in the delay-free case. Finally, the antithetic integral control of such networks is addressed. The controller is also extended to incorporate delays. It is shown that the ergodicity and output controllability conditions on the delayed reaction network are satisfied if and only if they are satisfied in the delay-free case.

All the obtained results can be straightforwardly extended to deal with uncertain reaction rates as in \cite{Briat:17LAA,Briat:16cdc,Briat:17ifacBio}. This was not done here because this extension is immediate and would not bring much to this paper. Note that it is unnecessary to consider uncertain reaction rates for the reactions in the delay lines as the results are independent of these parameters. Robustness with all delays is indeed automatically ensured whenever the (anyway necessary) condition that the delay-free network be ergodic is satisfied.

Extension to delay lines with finite capacity would be a very interesting topic to consider in order to consider the use of finite resources like in enzymatic networks \cite{Steiner:16} or some queueing processes that mRNA undergo while leaving the nucleus \cite{Hansen:18}. This can be done by incorporating \emph{counter species} counting the number of elements in the queue and inactivating the input reaction whenever the queue is full. The difficulty here resides in the fact that some of the queuing reactions become nonlinear and, hence, more difficult to consider. However, approximations could be helpful here in obtaining interesting results.

Another extension would be the study of the impact of delays on the variance and, in particular, the clarification of the conditions under which delays have a filtering effect on the noise. It may be tempting to say that this filtering mechanism relies on the randomness of the delay and would not be achieved by its deterministic counterpart. Elucidating this could help understand how cells filter their molecular noise and to design new synthetic filters topologies. Extrapolating a bit the result on gene expression, designing a filter that reduces noise may need to have small rates and, as a result, slow dynamics resulting in a slow response of the filter. In this regard, one can foresee the presence of some tradeoff between speed and quality of filtering.

The proposed approach allows one to consider almost deterministic constant delays in the sense that one can consider phase-type distributions that are arbitrarily close to the Dirac distribution without destroying the ergodicity of the delayed reaction network provided that its delay-free counterpart is ergodic. However, if constant deterministic delays need to be exactly incorporated, a new theory may be needed, which is definitely of interest. A starting point would be the consideration of a Markov jump processes on an infinite-dimensional state-space. The difficulty here will be the verification of the irreducibility property of the, now infinite-dimensional, state-space. A potential idea would be to look at reachability results for time-delay systems or, more generally, for infinite-dimensional systems described by partial differential equations. Regarding the verification of the positive recurrence property of the Markov process, the use of Foster-Lyapunov functions may not be suitable. Instead, we might need to consider more general functions or functionals. A possible starting point would be to look, again, at the literature on time-delay systems and adapt important tools such as Lyapunov-Razumikhin functions or Lyapunov-Krasovskii functionals to out present context.\\

As a final comment, we would like to propose the following conjecture:
\begin{conjecture}
  Any delayed unimolecular stochastic reaction network with constant deterministic/stochastic delays or time-varying stochastic delays is ergodic if and only if its delay-free counterpart is.
\end{conjecture}
We have proven under certain conditions that this conjecture is true when the delay is stochastic time-varying and phase-type distributed.

\section*{Acknowledgments}

This research has received funding from the European Research Council under the European Union's Horizon 2020 research and innovation programme / ERC grant agreement 743269 (CyberGenetics)

\bibliographystyle{plain}
%\bibliography{../../../../Lastbib/global,../../../../Lastbib/briat}

\end{document}